\documentclass[reqno]{amsart}

\usepackage{amsfonts}
\usepackage{amsmath}
\usepackage{amssymb}
\usepackage{amscd}
\usepackage{mathrsfs}  % This package allows the use of script letter. Type \mathscr to invoke math script.
\usepackage{amsbsy}
\usepackage{amsthm}
\usepackage{graphicx}
\usepackage{fancyhdr}
\usepackage{epsfig}
\usepackage{color}
\usepackage{xy}
\usepackage{enumitem}
\usepackage{CJKutf8}
\usepackage{inputenc}
\usepackage[encapsulated]{CJK}
\usepackage[CJK, overlap]{ruby}

\usepackage[OT2,OT1]{fontenc}  % This package allows the use of cyrillic fonts. They are invoked by typing \textcyr{...}
\newcommand\cyr{%
\renewcommand\rmdefault{wncyr}%
\renewcommand\sfdefault{wncyss}%
\renewcommand\encodingdefault{OT2}%
\normalfont \selectfont} \DeclareTextFontCommand{\textcyr}{\cyr}

\newcommand{\be}{\begin{equation}}
\newcommand{\ee}{\end{equation}}

\newcommand{\bes}{\begin{equation*}}
\newcommand{\ees}{\end{equation*}}

\newcommand{\inn}[2]{{\langle #1,#2 \rangle}}

\newcommand{\C}{\mathbb{C}}

\newcommand{\bH}{\mathbb{H}}

\newcommand{\N}{\mathbb{N}}

\newcommand{\R}{\mathbb{R}}

\newcommand{\Z}{\mathbb{Z}}

\newcommand{\cH}{\mathcal{H}}

\newcommand{\cO}{\mathcal{O}}

\newcommand{\sF}{\mathscr{F}}

\renewcommand{\Re}{\mathop{\mathrm{Re}}}
\renewcommand{\Im}{\mathop{\mathrm{Im}}}

\newcommand{\Span}{\mathrm{span\,}}
\newcommand{\clos}{\mathrm{clos}}

\newcommand{\nl}{\vskip 10pt\noindent}

\newcommand{\wh}[1]{{\widehat{#1}}}
\newcommand{\mcK}{\mathcal{K}}%
\newcommand{\mcH}{\mathcal{H}}%
\renewcommand{\rmdefault}{cmr} % Arial
\renewcommand{\sfdefault}{cmr} % Arial
\newtheorem{theorem}{Theorem}
\theoremstyle{plain}

\newtheorem{corollary}{Corollary}

\newtheorem{lemma}{Lemma}

\newtheorem{proposition}{Proposition}
\newtheorem{remark}{Remark}

\newtheorem*{notation*}{Notation}
\numberwithin{equation}{section}

\DeclareMathOperator\Sc{Sc}
\DeclareMathOperator\Ve{Vec}
\DeclareMathOperator\re{Re}
\DeclareMathOperator\im{Im}
\DeclareMathOperator\gr{graph}
\DeclareMathOperator\const{const.}

\newcommand{\mydot}{\,\cdot\,}

\begin{document}

\title[Quaternionic Fundamental Splines: Interpolation and Sampling]{Quaternionic Fundamental Cardinal Splines: Interpolation and Sampling}

\author{Jeffrey A. Hogan${}^{1,2,3}$}
\thanks{${}^1$Corresponding author}
\thanks{${}^2$Research partially supported by Australian Research Council grant DP160101537}
\address{School of Mathematical and Physical Sciences, Mathematics Bldg V123, University of Newcastle, University Drive, Callaghan NSW 2308, Australia}
\email{jeff.hogan@newcastle.edu.au}

\author{Peter R. Massopust${}^3$}
\thanks{${}^3$Research partially supported by Bavarian Research Alliance Grant BaylntAn\_UPA\_2017\_76}
\address{Centre of Mathematics, Research Unit M15, Technical University of Munich, Boltzmannstr. 3, 85748 Garching b. Munich, Germany}
\email{massopust@ma.tum.de}

\begin{abstract}
B-splines $B_{q}$ of quaternionic order $q$, for short quaternionic B-splines, are quaternion-valued  piecewise M\"{u}ntz polynomials whose scalar parts interpolate the classical Schoenberg splines $B_{n}$, $n\in\N$, with respect to degree and smoothness. They in general do not satisfy the interpolation property $B_{q}(n) = \delta_{n,0}$, $n\in\Z$. However, the application of the interpolation filter $(\sum\limits_{k\in\Z} \widehat{B_q}(\xi+2 \pi k))^{-1}$---if well-defined---in the frequency domain yields a cardinal fundamental spline of quaternionic order that  satisfies the interpolation property. We handle the ambiguity of the quaternion-valued exponential function appearing in the denominator of the interpolation filter and relate the filter to interesting properties of a quaternionic Hurwitz zeta function and the existence of complex quaternionic inverses. Finally, we show that the cardinal fundamental splines of quaternionic order fit into the setting of Kramer's Lemma and allow for a family of sampling, respectively, interpolation series.
\vskip 12pt\noindent
\textbf{Keywords and Phrases:} Quaternions, Clifford analysis, fundamental cardinal spline, Hilbert module, sampling, Kramer's lemma
\vskip 6pt\noindent
\textbf{AMS Subject Classification (2010):} 15A66, 30G35, 41A05, 42C40, 65D07, 94A20
\end{abstract}

\maketitle
\section{Introduction}\label{sec1}
In \cite{HMass}, a new class of B-splines with quaternionic order (quaternionic B-splines for short) were introduced with the intention to obtain interpolants and approximants for data that require a multi-channel description. For example, seismic data has four channels, each associated with a different kind of seismic wave: the so-called P (Compression), S (Shear), L (Love) and R (Rayleigh) waves. Similarly, the colour value of a pixel in a colour image has three components -- the red, green and blue channels. In order to perform the tasks of processing multi-channel signals and data, an appropriate set of analyzing basis functions is required. These basis functions should have the same analytic properties as those enjoyed by classical B-splines but should in addition be able to describe  multi-channel structures. In \cite{O1,O2}, a set of analyzing functions based on wavelets and Clifford-analytic methodologies were introduced in an effort to process four channel seismic data. A multiresolution structure for the construction of wavelets on the plane for the analysis of four-channel signals was outlined in \cite{HM}. 

The Schoenberg cardinal spline interpolation problem on the real line $\R$ consists of finding a spline $L$ with the property $L(k) = \delta_{k,0}$, for all $k\in \Z$, and which is a linear combination of shifts of a cardinal B-spline $B$: $L = \sum\limits_{k\in\Z} c_k B(\mydot - k)$. It is therefore natural  to ask whether the newly introduced quaternionic B-splines $B_q$ also solve a cardinal interpolation problem. In the classical case, only fundamental cardinal splines of even order exist and one should expect that some restrictions on $q$ in the quaternionic setting will occur as well.

The solution of the cardinal spline interpolation problem in the quaternionic setting is based on finding zero-free regions of a certain linear combination of quaternion-valued Hurwitz zeta functions. These zero-free regions are found by considering zero-free regions for associated complex-valued Hurwitz zeta functions. In the complex-valued case, such zero-free regions were derived in \cite{FGMS} and they will determine nonempty regions in the real quaternionic algebra $\mathbb{H}_\R$ for which the quaternionic cardinal spline interpolation problem has a unique solution.

Associated with the cardinal spline interpolation problem is Kramer's sampling theorem which gives conditions under which a function can be represented as an infinite series involving sampled values and a cardinal spline-type sampling function. We will extend Kramer's lemma to the complex quaternionic setting by first deriving a Riesz representation theorem for left Hilbert modules.

\section{Notation and Preliminaries}
The real, associative algebra of real quaternions ${\mathbb H}_{\mathbb R}$ is given by
\[
{\mathbb H}_{\mathbb R} :=\left\{a+\sum\limits_{i=1}^3v_ie_i : a, v_1, v_2, v_3\in{\mathbb R}\right\},
\]
where the imaginary units $e_1, e_2, e_3$ satisfy $e_1^2=e_2^2=e_3^2=-1$, $e_1e_2=e_3$, $e_2e_3=e_1$ and $e_3e_1=e_2$. Because of these  relations, ${\mathbb H}_\R$ is a non-commutative algebra. 

Each quaternion $q=a+\sum\limits\limits_{i=1}^3v_ie_i$ may be decomposed as $q=\Sc q + \Ve q$ where $\Sc q :=a$ is the {\it scalar part} of $q$ and $\Ve q :=v=\sum\limits\limits_{i=1}^3v_ie_i$ is the {\it vector part} of $q$. The {\it conjugate} $\overline{q}$ of the real quaternion $q=a+v$ is the quaternion $\overline{q}=a-v$. Note that $q\overline{q}=\overline{q}q=|q|^2=a^2+|v|^2=a^2+\sum\limits\limits_{i=1}^3v_i^2$. 
An element $q$ of $\bH_\R$ is called a \emph{pure quaternion} if $\Sc q = 0$. Note also that if $v=\sum\limits\limits_{j=1}^3v_je_j$ and $w=\sum\limits\limits_{j=1}^3w_je_j$ are pure quaternions, then 
\begin{equation}
vw=-\langle v,w\rangle +v\wedge w\label{vec mult},
\end{equation}
where $\langle v,w\rangle :=\sum\limits\limits_{j=1}^3v_jw_j$ is the scalar product of $v$ and $w$ and 
$$v\wedge w :=(v_2w_3-v_3w_2)e_1+(v_3w_1-v_1w_3)e_2+(v_1w_2-v_2w_1)e_3$$ 
is the vector (cross) product of $v$ and $w$.

If $q=q_0+\sum\limits\limits_{i=1}^3q_ie_i\in{\mathbb H}_{\mathbb C}:=\left\{a+\sum\limits\limits_{i=1}^3v_ie_i : a, v_1, v_2, v_3\in{\mathbb C}\right\}$ is a complex quaternion, we define the conjugate $\overline{q}$ of $q$ by $\overline q :=\overline{q_0}-\sum\limits\limits_{i=1}^3\overline{q_i}e_i$, where $\overline{q_i}$ is the complex conjugate of the complex number $q_i$. If $p=p_0+\sum\limits\limits_{i=1}^3p_ie_i\in{\mathbb H}_{\mathbb C}$, we define the inner product $\langle p,q\rangle$ to be the complex number  $\langle p,q\rangle :=\sum\limits\limits_{i=0}^3p_1\overline q_i=\Sc p\overline{q}$. We also define $e^q$ by the usual series: $e^q :=\sum\limits\limits_{j=0}^\infty \dfrac{q^j}{j!}$. We require the following bounds on $|e^q|$, which we state without proof.
\begin{lemma}\label{exp bounds} Let $q=a+v\in{\mathbb H}_{\mathbb R}$. Then we have
\begin{enumerate}
\item[1.] $|e^q|=e^a\leq e^{|q|}$.
\item[2.] If $z\in{\mathbb C}$ and $q'=zq\in{\mathbb H}_{\mathbb C}$, then $|e^{q'}|\leq e^{\sqrt 2|q'|}$.
\end{enumerate}
\end{lemma}

For $q=a+v=a+\sum\limits_{j=1}^3v_je_j\in{\mathbb H}_{\mathbb R}$ and
$z\in \C^* := {\mathbb C}\setminus\{0\}$, the quaternionic power $z^q\in{\mathbb H}_{\mathbb C}$ is defined by
\be\label{eq2.1}
z^q := z^a\bigg[\cos (|v|\log z)+\frac{v}{|v|}\sin (|v|\log z)\bigg].
\ee
This definition of $z^q$ allows for the usual differentiation and integration rules:
\begin{equation}
\frac{d}{dz} z^q = q z^{q-1}\qquad\text{and}\qquad \int z^q dz = \frac{z^{q+1}}{q+1} + \const, \quad q\neq -1.\label{diff/int formulae}
\end{equation}
In general, however, the semigroup property $z^{q_1}z^{q_2}=z^{q_1+q_2}$ fails to hold. (See, for instance, \cite{HMass}.)

If ${\mathbb F}\in \{{\mathbb R}, {\mathbb C}, {\mathbb H}_{\mathbb R},{\mathbb H}_{\mathbb C}\}$ and $1\leq p<\infty$, then $L^p({\mathbb R},{\mathbb F})$ denotes the Banach space of measurable functions $f:{\mathbb R}\to{\mathbb F}$ for which $\int\limits_{-\infty}^\infty |f(x)|^p\, dx<\infty$, where the meaning of $|f(x)|$ is dependent on ${\mathbb F}$. The Banach space $L^\infty ({\mathbb R},{\mathbb F})$ is defined similarly.  

On $L^2({\mathbb R},{\mathbb C})$, we define an inner product by  $\langle f,g\rangle =\int\limits_{-\infty}^\infty f(x)\overline{g(x)}\, dx$, and on $L^2({\mathbb R},{\mathbb H}_{\mathbb C})$ by
\begin{equation}
\langle f,g\rangle :=\Sc\bigg(\int_{-\infty}^\infty f(x)\overline{g(x)}\, dx\bigg)=\int_{-\infty}^\infty\langle f(x),g(x)\rangle\, dx.\label{IP complex quat}
\end{equation}

The  Fourier-Plancherel transform $\sF$ is defined on  $L^1 (\R, {\mathbb F})$ by 
\begin{gather*}
(\sF f) (\xi) := \wh{f} (\xi) := \int_\R f(x) e^{-i \xi x} dx
\end{gather*}
and may be extended to $L^2({\R},{\mathbb F})$, on which it becomes a multiple of a unitary mapping: $\langle\sF f,\sF g\rangle =2\pi \langle f,g\rangle$. 
\section{Quaternionic B-splines}
In this section, we state the definition of  B-splines of quaternionic order and briefly review some of their properties. More details and proofs can be found in \cite{HMass}.

The complex B-spline $B_z$ $(z\in{\mathbb C})$ is defined in the Fourier domain to be the function $\wh{B_z}:\R\to{\mathbb C}$ given by 
\[
\wh{B_z}(\xi )=\bigg(\dfrac{1-e^{-i\xi}}{i\xi}\bigg)^z, \quad \Re z > 1,
\]
and in the time domain by 
\[
B_z(t)=\dfrac{1}{\Gamma (z)}\sum_{k=0}^\infty\binom{z}{k}(t-k)_+^{z-1},\quad t\in{\mathbb R},\;\;\Re z >1.
\]
(See \cite{FBU}, \cite{HMass}.)

A B-spline $B_q$ of quaternionic order $q$ (for short {\it quaternionic B-spline}) is defined in the Fourier domain to be the function $\wh{B_q}:\R\to\bH_{\mathbb C}$ given by 
\begin{equation}
\wh{B_q} (\xi) := \left(\frac{1-e^{-i \xi}}{i \xi}\right)^q, \quad \Sc q > 1.\label{qB}
\end{equation}
Setting $\Xi (\xi) := \dfrac{1-e^{-i \xi}}{i \xi}$, 
we obtain the precise meaning of (\ref{qB}), namely,
\be\label{Bhat}
\wh{B_q} (\xi) = \Xi (\xi)^{\Sc q} \left(\cos (|v|\log \Xi (\xi))+\frac{v}{|v|}\sin (|v|\log \Xi (\xi))\right).
\ee
The function $\Xi$ has a removable singularity at $\xi = 0$ with $\Xi (0) = 1$.  It follows from $\re \Xi (\xi) = \xi^{-1}\,\sin \xi$ and $\im \Xi (\xi) = \xi^{-1}\,(1-\cos\xi)$, that $\gr\Xi\cap (\R^-\times \{0\}) = \emptyset$. Hence, $\Xi$ and therefore $\wh{B_q}$ are well-defined when $-\pi < \arg\Xi \leq \pi$. Equation \eqref{Bhat} also implies that $\wh{B_q} \in L^2(\R,\bH_{\mathbb C})$ for a fixed $q$ with $\Sc q > \frac12$ and in $L^1(\R,\bH_{\mathbb C})$ for a fixed $q$ with $\Sc q > 1$ as the fractional B-splines \cite{UB} satisfy these conditions. 
Note that $\wh{B_q}\in L^1(\R,\bH_{\mathbb C})$ implies that $B_q$ is uniformly continuous on $\R$.

The following results are a direct consequence of the corresponding properties of fractional B-splines \cite{UB}. To this end, for a real number $s\geq 0$ and $1\leq p\le \infty$,  the Bessel potential space $H^{s,p}({\mathbb R},{\mathbb H}_\C)$ is given by
\[
H^{s,p}({\mathbb R},{\mathbb H}_{\mathbb C}):=\left\{f\in L^p({\mathbb R},{\mathbb H}_{\mathbb C}) : \sF^{-1}[(1+|\xi |^2)^{s/2}\sF f]\in L^p({\mathbb R},{\mathbb H}_{\mathbb C})\right\}.
\]
\begin{proposition} 
Let $B_q$ be a quaternionic B-spline with $\Sc q >\frac12$. Then $B_q$ enjoys the following properties:
\begin{enumerate}
\item[\emph{(i)}] {Decay}: $B_q \in \cO(|t|^{-\Sc q-1})$ as $|t|\to \infty$.
\item[\emph{(ii)}] {Smoothness}: $B_q\in H^{s,p} (\R,\bH_{\mathbb C})$ for $1\leq p\leq\infty$ and $0\leq s < \Sc q + \frac1p$.
\item[\emph{(iii)}] {Reproduction of Polynomials}: $B_q$ reproduces polynomials up to order $\lceil\Sc q\rceil$, where the ceiling function $\lceil\mydot\rceil:\R\to \Z$ is given by $r\mapsto \min\{n\in \Z : n\geq r\}$.
\end{enumerate}
\end{proposition}

Of particular interest in this paper is property (i) of the proposition. A proof can be obtained from \cite{UB}. Here, we give a less elaborate argument that gives a weaker decay estimate which is nonetheless sufficient for our purposes.

With $D :=\dfrac{1}{i}\dfrac{d}{d\xi}$ and $m$ a non-negative integer, Leibnitz's formula gives
$$D^m\widehat{B_q}(\xi )=i^{-m}\sum_{k=0}^m\binom{m}{k}c_{q,m,k}(1-e^{-i\xi})^{q-k}\xi^{-q-m+k}$$
with $c_{q,m,k}$ $(0\leq k\leq m)$ quaternionic constants. Therefore, for all integers $\ell\geq 1$,
\begin{align*}
\int_{2\pi (\ell -1/2)}^{2\pi (\ell +1/2)}|D^m\widehat{B_q}(\xi )|\, d\xi&\leq\sum_{k=0}^mc_{q,m,k}\int_{2\pi (\ell -1/2)}^{2\pi (\ell +1/2)}|1-e^{-i\xi}|^{\Sc q-k}|\xi |^{-\Sc q-m+k}\\
&\leq\sum_{k=0}^mc_{q,m,k}(\ell -1/2)^{-\Sc q-m+k}\int_{2\pi (\ell -1/2)}^{2\pi (\ell +1/2)}|1-e^{-i\xi}|^{\Sc q-k}\, d\xi\\
&\leq\sum_{k=0}^mc_{q,m,k}'(\ell -1/2)^{-\Sc q-m+k}\leq c_{q,m}(\ell -1/2)^{-\Sc q}
\end{align*}
provided $\Sc q -m>-1$. Therefore,
$$\int_{-\infty}^\infty |D^m\widehat{B_q}(\xi )|\, d\xi\leq 2c_{q,m}\sum_{\ell =1}^\infty (\ell -1/2)^{-\Sc q}<\infty$$
provided $\Sc q>1$ and $m<\Sc q+1$. Hence, for this range of values of $q$ and $m$ we have $B_q(t)\leq c_m|t|^{-m}$ for all $t\in{\mathbb R}$. Thus, we have proved the following result.

\begin{proposition} Let $\Sc q>1$, then 
\begin{enumerate}
\item[(i)] $B_q\in L^1({\mathbb R})$.
\item[(ii)] The Fourier series $\sum\limits_{k=-\infty}^\infty B_q(k)e^{ik\xi}$ is absolutely convergent.
\end{enumerate}
\end{proposition}

The time domain representation of the quaternionic B-spline $B_q$ is given by
\be
B_q (t) =  \frac{1}{\Gamma (q)}\,\sum_{k=0}^\infty (-1)^k \binom{q}{k} (t - k)_+^{q-1}, \quad t\in \R,\;\;\Sc q > 1.\label{time domain}
\end{equation}
This equality holds in the sense of distributions and in $L^2 (\R)$. Here, the quaternionic binomial coefficient is defined via the quaternionic Gamma function:
\be\label{gamma}
\Gamma (q) := \int_0^\infty t^{a-1}\cos (|v|\log t) e^{-t} dt + \frac{v}{|v|}\,\int_0^\infty t^{a-1}\sin (|v|\log t) e^{-t} dt.
\ee 
It was shown in \cite{HMass} that $\Gamma (q)$ can be written as
\begin{align}\label{qGamma}
\Gamma (q) & = \frac{\Gamma (\sigma + i |v|) + \Gamma (\sigma - i|v|)}{2} + \frac{v}{|v|} \frac{(\Gamma (\sigma + i |v|) - \Gamma (\sigma - i|v|)}{2i}\nonumber\\
& = \Re \Gamma (z) + \frac{v}{|v|} \Im \Gamma(z),\quad z := \sigma + i |v|.
\end{align}
The next result shows that there exists also a structure formula for the quaternionic B-splines.
\begin{theorem}\label{thm1}
If $q=a+v\in{\mathbb H}_{\mathbb R}$ and $z=a+i|v|\in{\mathbb C}$, then
\[
B_q(t)=\Re B_z(t) + \frac{v}{|v|}\Im B_z(t),\qquad t\in \R\setminus\{0\}.
\]
\end{theorem}
\begin{proof} With $q\in {\mathbb H}_{\mathbb R}$ and $z\in{\mathbb C}$ as in the statement of the Theorem, in \cite{HMass} the quaternionic Pochhammer symbol $(q)_j=q(q-1)\dots (q-j+1)$ was decomposed as \begin{equation}
(q)_j=\Re (z)_j+\frac{v}{|v|}\Im (z)_j.\label{q Poch}
\end{equation}
By (\ref{q Poch}) we have 
\[
\binom{q}{k}=\dfrac{(q)_k}{k!}=\dfrac{\Re (z)_k + \frac{v}{|v|}\Im (z)_k}{k!}=\Re\binom{z}{k}+\dfrac{v}{|v|}\Im\binom{z}{k}.
\]
Furthermore, 
\begin{align*}
t^{q-1}&=t^{a-1}[\cos (|v|\log t)+\dfrac{v}{|v|}\sin (|v|\log t)]\\
&=t^{a-1}[\Re t^{i|v|} + \dfrac{v}{|v|}\Im t^{i|v|}]=\Re t^{z-1}+\dfrac{v}{|v|}\Im t^{z-1}.
\end{align*}
Thus,
\begin{align}
\binom{q}{k}(t-k)_+^{q-1}&=\bigg[\Re\binom{z}{k}+\dfrac{v}{|v|}\Im\binom{z}{k}\bigg]\bigg[\Re ((t-k)_+^{z-1})+\frac{v}{|v|}\Im ((t-k)_+^{z-1})\bigg]\notag\\
&=\Re\binom{z}{k}\Re (t-k)_+^{z-1} - \Im\binom{z}{k}\Im (t-k)_+^{z-1}\notag\\
&+\frac{v}{|v|}\bigg[\Re\binom{z}{k}\Im (t-k)_+^{z-1} + \Im \binom{z}{k}\Re (t-k)_+^{z-1}\bigg]\notag\\
&=\Re \bigg[\binom{z}{k}(t-k)_+^{z-1}\bigg]+\frac{v}{|v|}\Im\bigg[\binom{z}{k}(t-k)_+^{z-1}\bigg].\label{qkt}
\end{align} 
Substituting (\ref{qkt}) into (\ref{time domain}) yields
\begin{align}
B_q(t)
&=\frac{1}{\Gamma (q)}\bigg[\Re\bigg(\sum_{k=0}^{\infty}\binom{z}{k}(t-k)_+^{z-1}\bigg)+\frac{v}{|v|}\Im\bigg(\sum_{k=0}^{\infty}\binom{z}{k}(t-k)_+^{z-1}\bigg)\bigg]\notag\\
&=\frac{1}{\Gamma (q)}[\Re (\Gamma (z)B_z(t))+\frac{v}{|v|}\Im (\Gamma (z) B_z(t))].\label{B_q complex}
\end{align}
On the other hand, using \eqref{qGamma}, we obtain
\begin{equation}
\frac{1}{\Gamma (q)}=\frac{\Re \Gamma (z) - \frac{v}{|v|}\Im \Gamma (z)}{|\Gamma (z)|^2},\label{inv Gamma}
\end{equation}
and substituting (\ref{inv Gamma}) into (\ref{B_q complex}) gives
\begin{align*}
B_q(t)&=\bigg[\frac{\Re \Gamma (z) - \frac{v}{|v|}\Im \Gamma (z)}{|\Gamma (z)|^2}\bigg][\Re (\Gamma (z)B_z(t))+\frac{v}{|v|}\Im (\Gamma (z) B_z(t))]\\
&=\frac{1}{|\Gamma (z)|^2}\bigg[[\Re \Gamma (z)\Re (\Gamma (z)B_z(t))+\Im \Gamma (z)\Im (\Gamma (z)B_z(t))]\\
&\qquad\qquad +\frac{v}{|v|}[\Re \Gamma (z)\Im (\Gamma (z)B_z(t))-\Im \Gamma (z)\Re (\Gamma (z)B_z(t))]\bigg]\\
&=\frac{1}{|\Gamma (z)|^2}\bigg[\Re (|\Gamma (z)|^2B_z(t))+\frac{v}{|v|}\Im (|\Gamma (z)|^2B_z(t))\bigg]\\
&=\Re B_z(t) + \frac{v}{|v|}\Im B_z(t).\qedhere
\end{align*}
\end{proof}
\section{Quaternion Inverses}
Let $q=z+\sum\limits_{i=1}^3w_ie_i\in{\mathbb H}_{\mathbb C}$ (i.e., $z,w_1,w_2,w_3\in{\mathbb C}$). Then with $\tilde q=z-w$ we have $q\tilde q=z^2+\sum\limits_{i=1}^3w_i^2$, so that $q$ is invertible if and only if $z^2+\sum\limits_{i=1}^3w_i^2\neq 0$ and in this case $q^{-1}=\tilde q/(z^2+\sum\limits_{i=1}^3w_i^2)$.
When $q=a+v\in {\mathbb H}_{\mathbb R}$, $q$ is invertible provided $q\neq 0$ and then $q^{-1}=\overline{q}/|q|^2$.

\begin{lemma} If $\lambda\in{\mathbb C}$ and $q\in{\mathbb H}_{\mathbb R}$ then $e^{\lambda q}$ is invertible in ${\mathbb H}_{\mathbb C}$ and
$$(e^{\lambda q})^{-1}=e^{-\lambda q}.$$
\end{lemma}

\begin{proof} By direct calculation we find that 
$$e^{\lambda q}=e^{\lambda a}e^{\lambda v}=e^{\lambda a}\bigg(\cos (\lambda |v|)+\frac{v}{|v|}\sin (\lambda |v|)\bigg)$$
and therefore,
\begin{align*}
e^{\lambda q}e^{-\lambda q}&=e^{\lambda a}\bigg(\cos (\lambda |v|)+\frac{v}{|v|}\sin (\lambda |v|)\bigg)e^{-\lambda a}\bigg(\cos (\lambda |v|)-\frac{v}{|v|}\sin (\lambda |v|)\bigg)\\
&=\cos^2(\lambda |v|)+\sin^2(\lambda |v|)=1.\qedhere
\end{align*}
\end{proof}

\begin{lemma}\label{lem3} 
For all $t\in {\mathbb R}^+$ and $q\in {\mathbb H}_{\mathbb R}$, we have
$$(-t)^q=e^{i\pi q}t^q.$$
\end{lemma}
\begin{proof} From the definition of a quaternionic power, we have for $t > 0$,
\begin{align*}
(-t)^q&=(-t)^a\bigg[\cos (|v|\log (-t))+\frac{v}{|v|}\sin (|v|\log (-t))\bigg]\\
&=(-t)^a\bigg[\cos (|v|(\log t +i\pi ))+\frac{v}{|v|}\sin (|v|(\log t +i\pi ))\bigg]\\
&=(-t)^a\bigg[\cos (|v|\log t)\cosh (\pi |v|)-i\sin (|v|\log t)\sinh (\pi |v|)\\
&\qquad\qquad\qquad\qquad +\frac{v}{|v|}(\sin (|v|\log t)\cosh (\pi |v|)+i\cos (|v|\log t)\sinh (\pi |v|))\bigg]\\
&=(-t)^a\bigg(\cos (|v|\log t)\bigg[\cosh (\pi |v|)+\frac{iv}{|v|}\sinh (\pi |v|)\bigg]\\
&\qquad\qquad\qquad\qquad +\sin (|v|\log t)\bigg[-i\sinh (\pi |v|)+\frac{v}{|v|}\cosh (\pi |v|)\bigg]\bigg)\\
&=e^{i\pi a}e^{i\pi v}t^a\bigg(\cos (|v|\log t)+\frac{v}{|v|}\sin (|v|\log t)\bigg)=e^{i\pi q}t^{a+v}=e^{i\pi q}t^q.
\end{align*}
Here, we used the usual convention that $t^a := e^{i\pi a}|t|^a$, for $t<0$.
\end{proof}
\section{A Quaternionic Zeta Function}
In this section we study  the quaternionic analog $\zeta (q,a)$ of the classical Hurwitz zeta function $\zeta (s,a)$, which is defined by
\[
\zeta (s, a) := \sum_{k=0}^\infty \frac{1}{(a+n)^s}, \quad \re s > 1\text{ and } \re a > 0.
\]
To do so, we first define functions $\chi_+$ and $\chi_-$ by
$$\chi_\pm (v) :=\frac{1}{2}\bigg(1\pm i\frac{v}{|v|}\bigg),\qquad \bigg(v=\sum_{i=1}^3v_ie_i\in{\mathbb H}_{\mathbb R}\bigg)$$
and note that 
\begin{equation}
\chi_\pm (v)^2=\chi_\pm (v);\quad \chi_+(v)\chi_-(v)=\chi_-(v)\chi_+(v)=0.\label{chi props}
\end{equation}

Given $\sigma >1$ and $q=\sigma+v\in{\mathbb H}_{\mathbb R}$, we define the quaternionic zeta function $\zeta (q,\mydot) :{\mathbb R}\to {\mathbb H}_{\mathbb R}$ by
$$
\zeta (q, a) :=\sum\limits_{k=0}^\infty\frac{1}{(a+k)^q},\qquad a > 0.
$$
\begin{lemma} Let $q=a+v\in{\mathbb H}_{\mathbb R}$. Then with $w=\sigma+i|v|\in{\mathbb C}$, we have
$$
\zeta (q, a) =\chi_+(v)\zeta (w,a)+\chi_-(v)\zeta({\overline w},a),
$$
where $\zeta (w,a)$ and $\zeta({\overline w},a)$ denote the complex-valued Hurwitz zeta functions. 
Further, for all positive integers $n$,
$$\zeta (q,a)^n=\chi_+(v)\zeta(w,a)^n+\chi_-(v)\zeta({\overline w},a)^n.$$
\end{lemma}

\begin{proof} 
Note that \eqref{eq2.1} and the assumption $\Sc q = \sigma > 1$ imply that the infinite series $\sum\limits_{k=0}^\infty\dfrac{1}{(a+k)^q}$, $a > 0$, converges absolutely. For each $z\in{\mathbb C}^*$ we have 
\begin{align*}
z^q&=z^a\bigg[\cos (|v|\log z)+\frac{v}{|v|}\sin (|v|\log z)\bigg]\\
&=z^a\bigg[\frac{e^{i|v|\log z}+e^{-i|v|\log z}}{2}+\frac{v}{|v|}\bigg(\frac{e^{i|v|\log z}-e^{-i|v|\log z}}{2i}\bigg)\bigg]\\
&=z^a[\chi_-(v)z^{i|v|}+\chi_+(v)z^{-i|v|}]=\chi_-(v)z^w+\chi_+(v)z^{\overline{w}}
\end{align*}
where $w=a+i|v|\in{\mathbb C}$. Thus,
\begin{align*}
\zeta(q,a)&=\sum\limits_{k=0}^\infty (a+k)^{-q}\\
&=\sum\limits_{k=0}^\infty\bigg[\frac{\chi_-(v)}{(a+k)^w}+\frac{\chi_+(v)}{(a+k)^{\overline{w}}}\bigg]=\chi_-(v)\zeta(w,a)+\chi_+(v)\zeta({\overline{w}},a).
\end{align*}
Hence, because of (\ref{chi props}), we have
\begin{align*}
\zeta(q,a)^n&=\sum\limits_{j=0}^n\binom{n}{j}\chi_-(v)^j\zeta(w,a)^j\chi_+(v)^{n-j}\zeta({\overline{w}},a)^{n-j}\\
&=\chi_-(v)\zeta(w,a)^n+\chi_+(v)\zeta({\overline w},a)^n.\qedhere
\end{align*}
\end{proof}
\section{Interpolation with Quaternionic B-Splines}
In order to solve the cardinal spline interpolation problem using the classical Curry-Schoenberg splines \cite{chui,schoenberg}, one constructs a fundamental cardinal spline function which is a linear bi-infinite combination of polynomial B-splines $B_n$ of fixed order $n\in \N$ that interpolates the data set $\{\delta_{m,0} : m\in \Z\}$. More precisely, one solves the bi-infinite system
\[
\sum_{k\in \Z} c_k^{(n)} B_n \left(\frac{n}{2} + m - k\right) = \delta_{m,0},\quad m\in \Z,
\]
for the sequence $\{c_k^{(n)} : k\in \Z\}$ and then defines the fundamental cardinal spline $L_n:\R\to\R$ of order $n\in \N$ by
$L_n(x)=\sum\limits_{k\in \Z} c_k^{(n)} B_n \left(\frac{n}{2} + x - k\right)$.
The Fourier transform of  $L_n$ is given  by
\be\label{fundspline}
\wh{L}_n (\xi) = \frac{\wh{\tau_{-n/2}{B}_n}(\xi)}{\displaystyle{\sum_{k\in \Z}}\,\wh{\tau_{-n/2}{B}_n}(\xi + 2\pi k)}
\ee
where $\tau_\alpha f(x)=f(x-\alpha)$ $(x,\alpha \in{\mathbb R})$. Using the Euler-Frobenius polynomials associated with the B-splines $B_n$, one can show that the denominator in (\ref{fundspline}) does not vanish on the unit circle. For details, see for instance \cite{chui, schoenberg}.

One of our goals is to construct a {\em fundamental cardinal spline $L_q:\R\to \bH_\R$  of quaternionic order $q = a + v$}, $a > 1$, of the form
\be\label{compint2}
L_q := \sum_{k\in \Z} c_k^{(q)} B_q \left(\mydot - k\right),\quad a > 1,
\ee
satisfying the interpolation problem
\be\label{complexint2}
L_q (m) = \delta_{m,0}, \quad m\in \Z, 
\ee
for an appropriate bi-infinite quaternion-valued sequence $\{c_k^{(q)} : k\in \Z\}$ and for appropriate $q$ belonging to some nonempty open subset of $\bH_\R$. The case $q\in \C$ reduces to the setting for complex B-splines and the existence of fundamental cardinal splines for complex B-splines was investigated and proven in \cite{FM,FGMS}.

Taking the Fourier transform of \eqref{compint2} and applying \eqref{complexint2} to eliminate the expression containing the unknowns $\{c_k^{(q)} : k\in \Z\}$ gives (formally) a formula for $L_q$ similar to (\ref{fundspline}):
\be\label{compfundspline}
\wh{L}_q (\xi) = \frac{\wh{B}_q (\xi)}{{\sum\limits_{k\in \Z}}\,B_q (k)\,z^k},\quad z = e^{i\xi}, \;\;\xi\in \R.
\ee
Equation \eqref{compfundspline} contains a complex quaternion in the denominator. The next result gives necessary and sufficient conditions for the complex quaternion $\sum\limits_{k\in \Z}\,B_q (k)\,z^k$ to have an inverse. In the following, we write $\sum\limits_k$ to mean $\sum\limits_{k\in \Z}$.

\begin{theorem} Let $q=a+v\in{\mathbb H}_{\mathbb R}$ and $w=a+i|v|\in{\mathbb C}$. Further, let $|z| = 1$. Then $\sum\limits_\ell B_q(\ell )z^\ell$ is invertible in ${\mathbb H}_{\mathbb C}$ if and only if $\left(\sum\limits_\ell B_w(\ell )z^\ell \right)\left(\sum\limits_k B_{\overline w}(k)z^k\right)\neq 0$. In this case we have
\be\label{qinverse}
\bigg(\sum\limits_\ell B_q(\ell )z^\ell\bigg)^{-1}=\frac{\sum\limits_\ell\left[\Re B_w (\ell ) - \frac{v}{|v|}\Im B_w(\ell )\right]z^\ell}{\left(\sum\limits_\ell B_w(\ell )z^\ell \right)\left(\sum\limits_\ell{B_{\overline w}(\ell )}z^\ell\right)}.
\ee
\end{theorem}

\begin{proof} First note that if $q=a+v\in{\mathbb H}_{\mathbb R}$ and $(w)_k$ is the Pochhammer symbol of the complex number $w=a+i|v|$ then an application of Theorem \ref{thm1} gives 
\begin{align*}
&\Gamma(q) \sum\limits_\ell  B_q(\ell )z^\ell =\sum\limits_\ell\sum\limits_k(-1)^k\frac{(q)_k}{k!}(\ell -k)_+^qz^\ell\\
%&=\sum\limits_{\ell ,k}\frac{(-1)^k}{k!}\bigg(\Re(w)_k+\frac{v}{|v|}\Im (w)_k\bigg)(\ell -k)_+^qz^\ell\\
&=\sum\limits_{\ell ,k}\frac{(-1)^k}{k!}\bigg(\Re (w)_k+\frac{v}{|v|}\Im (w)_k\bigg)(\ell -k)_+^a \bigg[\cos (|v|\log (\ell -k)_+)\\
&\qquad\qquad\qquad\qquad\qquad\qquad\qquad\qquad\qquad\qquad\qquad+ \left.\frac{v}{|v|}\sin (|v|\log (\ell -k)_+)\right]z^\ell\\
&=\sum\limits_{\ell ,k}\frac{(-1)^k}{k!}(\ell -k)_+^a \bigg([\Re (w)_k\cos (|v|\log (\ell -k)_+)-\Im (w)_k\sin (|v|\log (\ell -k)_+)]\\
&\qquad\qquad\qquad +\frac{v}{|v|}[\Im (w_k)\cos (|v|\log (\ell -k)_+)+\Re (w)_k\sin (|v|\log (\ell -k)_+)]\bigg)z^\ell\\
&=\sum\limits_{\ell ,k}\frac{(-1)^k}{k!}(\ell -k)_+^a \left[\Re ((w)_ke^{i|v|\log (\ell -k)_+})+\frac{v}{|v|}\Im ((w)_ke^{i|v|\log (\ell -k)_+})\right]z^\ell\\
&=\sum\limits_\ell \left[\Re \left(\sum\limits_k\frac{(-1)^k}{k!}(\ell -k)_+^a (w)_k(\ell -k)_+^{i|v|}\right)\right. \\
&\qquad\qquad\qquad\quad+ \left.\frac{v}{|v|}\Im \left(\sum\limits_k\frac{(-1)^k}{k!}(\ell -k)_+^a (w)_k(\ell -k)_+^{i|v|}\right)\right]z^\ell\\
&=\sum\limits_\ell \left[\Re (\Gamma (w) B_w (\ell ))+\frac{v}{|v|}\Im (\Gamma (w) B_w(\ell ))\right]z^\ell\\
&= \sum\limits_\ell \left[\Re (\Gamma (w) + \frac{v}{|v|}\Im (\Gamma (w))\right] \left[\Re (B_w (\ell ))+\frac{v}{|v|}\Im (B_w(\ell ))\right]z^\ell\\
&= \Gamma (q) \sum\limits_\ell \left[\Re (B_w (\ell ))+\frac{v}{|v|}\Im (B_w(\ell ))\right]z^\ell,
\end{align*}
where we used \eqref{qGamma}.
Let 
$$
Z=\sum\limits_\ell \text{Re}(B_w(\ell ))z^\ell =\dfrac{1}{2}\sum\limits_\ell \left[B_w (\ell )z^\ell +\overline{B_w(\ell )}z^\ell\right]=\frac{1}{2}(A+B)$$ 
and, for $i\in\{1,2,3\}$, 
$$
V_i=\dfrac{v_i}{|v|}\sum\limits_\ell \text{Im}((B_w(\ell ))z^\ell =\frac{1}{2i}\frac{v_i}{|v|}\left[\sum\limits_\ell B_w(\ell )z^\ell -\sum\limits_\ell\overline{B_w(\ell )}z^\ell \right]=\frac{1}{2i}\frac{v_i}{|v|}(A-B)
$$
so that $\sum\limits_\ell B_q(\ell )z^\ell =Z+\sum\limits_{i=1}^3V_ie_i$. Then 
\begin{equation*}
Z^2+\sum\limits_{i=1}^3V_i^2=\dfrac{1}{4}(A+B)^2-\frac{1}{4}(A-B)^2=AB=\left(\sum\limits_\ell B_w(\ell )z^\ell \right)\left(\sum\limits_\ell\overline{B_w(\ell )}z^\ell\right).
\end{equation*}
However for $w\in{\mathbb C}$ with Pochhammer symbol $(w)_j$ we have $(\overline{w})_j=\overline{(w)_j}$ and if $t$ is a real number, $\overline{t^w}=t^{\overline{w}}$. Therefore,
$$
\overline{B_w(t)}=\overline{\frac{1}{\Gamma(w)}\sum\limits_k\frac{(-1)^k}{k!}(w)_j (t-k)_+^w} = \frac{1}{\Gamma(\overline{w})}\sum\limits_k\frac{(-1)^k}{k!}(\overline{w})_j(t-k)_+^{\overline{w}}=B_{\overline{w}}(t),
$$
so that 
$$
Z^2+\sum\limits_{i=1}^3V_i^2=\left(\sum\limits_\ell B_w(\ell)z^\ell\right)\left(\sum\limits_\ell B_{\overline{w}}(\ell )z^\ell \right).
$$
Thus, $\sum\limits_\ell B_q(\ell )z^\ell$ is invertible if and only if $\left(\sum\limits_\ell B_w(\ell)z^\ell\right)\left(\sum\limits_\ell B_{\overline{w}}(\ell )z^\ell \right)\neq 0$. In this case the inverse is given by
\begin{align*}
\bigg(\sum\limits_\ell B_q(\ell )z^\ell\bigg)^{-1} &=\frac{\left(\sum\limits_\ell B_q(\ell )z^\ell \right)\tilde{\phantom{l}}}{\left(\sum\limits_\ell B_w(\ell )z^\ell \right)\left(\sum\limits_\ell\overline{B_w(\ell )}z^\ell\right)}\\ \\
&=\frac{\sum\limits_\ell \left[\Re B_w (\ell ) - \frac{v}{|v|}\Im B_w(\ell )\right]z^\ell}{\left(\sum\limits_\ell B_w(\ell )z^\ell \right)\left(\sum\limits_\ell B_{\overline{w}}(\ell )z^\ell\right)}.\qedhere
\end{align*}
\end{proof}

By the Poisson summation formula, we have
$\sum\limits_\ell B_w(\ell )z^\ell = \sum\limits_\ell \wh{B}_w (z + 2\pi\ell)$. It was shown in \cite{FGMS} that $\sum\limits_\ell \wh{B}_w (z + 2\pi\ell) \neq 0$ iff
\be\label{hurwitz}
\zeta (w,\alpha) + e^{-i \pi w}\,\zeta (w, 1 - \alpha)\neq 0,
\ee
where we take the principal value of the multi-valued function $e^{-i \pi (\mydot)}$ and where $\zeta(w,\alpha)$ denotes the Hurwitz zeta function with parameter $\frac{\xi}{2\pi} =: \alpha \in (0,1)$. 

In \cite{FGMS}, regions $R\subset{\mathbb C}$ %for $w = a + i|v|$, $a > 2$, were constructed 
for which $w = a + iv \in R$, $a >2$, implies that the Hurwitz zeta function \eqref{hurwitz} is zero-free were constructed. These regions satisfy  $\overline{R} = R$ (i.e., $R$ is symmetric with respect to the real axis) and $R$ is either a rectangular region centered at points $2n\in \C$, $n\in \N$, of width less than one and strictly positive height, or an annular-shaped region centered at points $2n\in \C$, $n\in \N$, of width less than one and angular extend $0 \neq \vartheta\in (-\pi,\pi)$.
Let 
\begin{equation}
Q_R := \{q = a+ v\in \bH_\R : a + i |v|\in R\}.\label{region}
\end{equation}
Then, since $\overline{R} = R$, the denominator of \eqref{qinverse} is nonzero provided $q\in Q_R$. Applying the Poisson summation formula to the denominator of \eqref{compfundspline} yields
\[
\wh{L}_q (\xi) = \frac{\wh{B}_q (\xi)}{{\sum\limits_{k\in \Z}}\,\wh{B}_q (\xi + 2\pi k)},\quad\xi\in \R.
\]
We note that $\wh{L}_q(0) = 1$ and $\wh{L}_q(2\pi) = 0$. As the expression in the denominator is $2\pi$-periodic, it suffices to assume that $0 < \xi < 2\pi$. Inserting \eqref{qB} into the above expression for $\wh{L}_q$ and using the facts that $\left(\dfrac{z_1}{z_2}\right)^q = \dfrac{z_1^q}{z_2^q}$, $z_1, z_2\in \C$ and $(i t)^q = i^q t^q$, $t\in \R$, with the convention $\arg t = - \pi $ for $t<0$, gives
\begin{align*}
\wh{L}_q (\xi) = \frac{\left(\frac{1-e^{-i \xi}}{i \xi}\right)^q}{ \sum\limits_{k\in \Z}\,\left(\frac{1-e^{-i (\xi + 2\pi k)}}{i (\xi + 2\pi k)}\right)^q} = \frac{1/\xi^q}{{\sum\limits_{k\in \Z}}\, \frac{1}{(\xi + 2\pi k)^q}}.
\end{align*}
Setting $\alpha : = \frac{\xi}{2\pi}\in (0,1)$, the sum in the above denominator can be rewritten in the form
\begin{eqnarray}
\sum_{k\in \Z}\, \frac{1}{(k + \alpha)^q}  & = & \sum_{k=0}^\infty\, \frac{1}{(k + \alpha)^q} + \sum_{k=0}^\infty\, \frac{1}{(\alpha - 1 - k)^q} 
 \nonumber\\
& = &\sum_{k=0}^\infty\, \frac{1}{(k + \alpha)^q} + e^{-i \pi q}\sum_{k=0}^\infty\, \frac{1}{(k + 1 - \alpha)^q}
\nonumber\\
& = &\zeta (q,\alpha) + e^{-i \pi q}\,\zeta (q, 1 - \alpha),
\label{Zerlegung in zetas}
\end{eqnarray}
where we have used Lemma \ref{lem3}. With the above observations, we obtain the next result which provides zero-free regions for the quaternionic zeta function.

\begin{proposition}
Let $Q_R$ be defined as in (\ref{region}). The combination 
\[
\zeta (q,\alpha) + e^{-i \pi q}\,\zeta (q, 1 - \alpha)
\]
of quaternionic Hurwitz zeta functions is zero-free for all $\alpha\in (0,1)$ provided $q\in Q_R$.
\end{proposition} 

Combining the above results and observations yields the following theorem.

\begin{theorem}
Suppose that ${B}_{q}$ is a quaternionic B-spline with $q\in Q_R$. 
Then
\begin{align}\label{L}
{L}_q (x) & := \frac{1}{2\pi}\,\int\limits_{\R} \frac{\xi^{-q}\,e^{i \xi x}}{\zeta (q, \xi/2\pi) + e^{-i \pi q}\zeta (q, 1 - \xi/2\pi)}\,d\xi, \quad x\in \R,
\end{align}
is a fundamental interpolating spline of quaternionic order $q$ in the sense that
\[
{L}_q ({m}) = \delta_{{m},0}, \quad \text{for all }{m}\in \Z.
\]
The Fourier inverse in \eqref{L} holds in the $L^1$- and $L^2$-sense.
\end{theorem}

In Figures \ref{fig1}--\ref{fig5}, two fundamental cardinal splines with $q_1 = 6.2 + \frac{1}{2\sqrt{2}} e_1 - \frac14 e_2 + \frac14 e_3$ and $q_2 = 2.5 + \frac{1}{4\sqrt{2}} e_1 + \frac18 e_2 - \frac{\sqrt{13}}{8} e_3$ are depicted.
The graphs of these functions were constructed by sampling the integrand in Eqn.~\eqref{L} and performing an inverse Fourier transform on this set of sampled values.

\begin{figure}[h!]
\begin{center}
\includegraphics[width=4cm, height= 3cm]{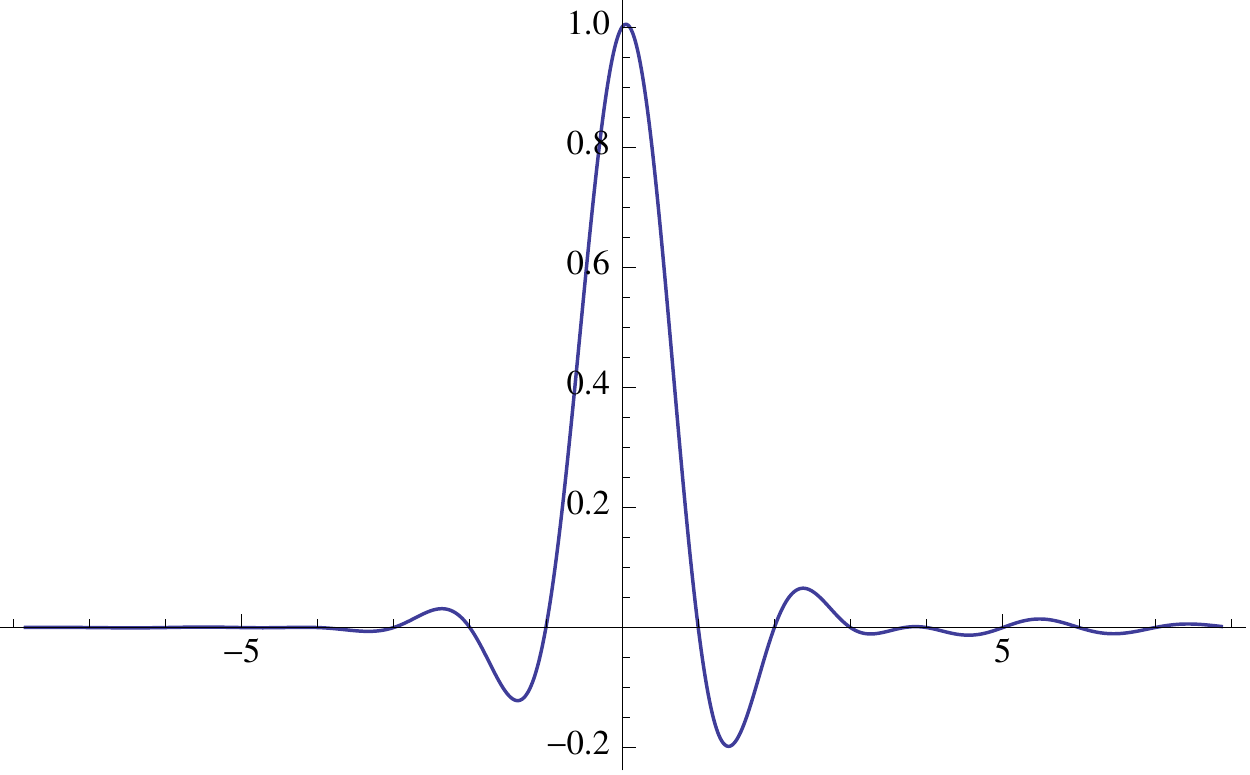}\hspace*{1.5cm}
\includegraphics[width=4cm, height= 3cm]{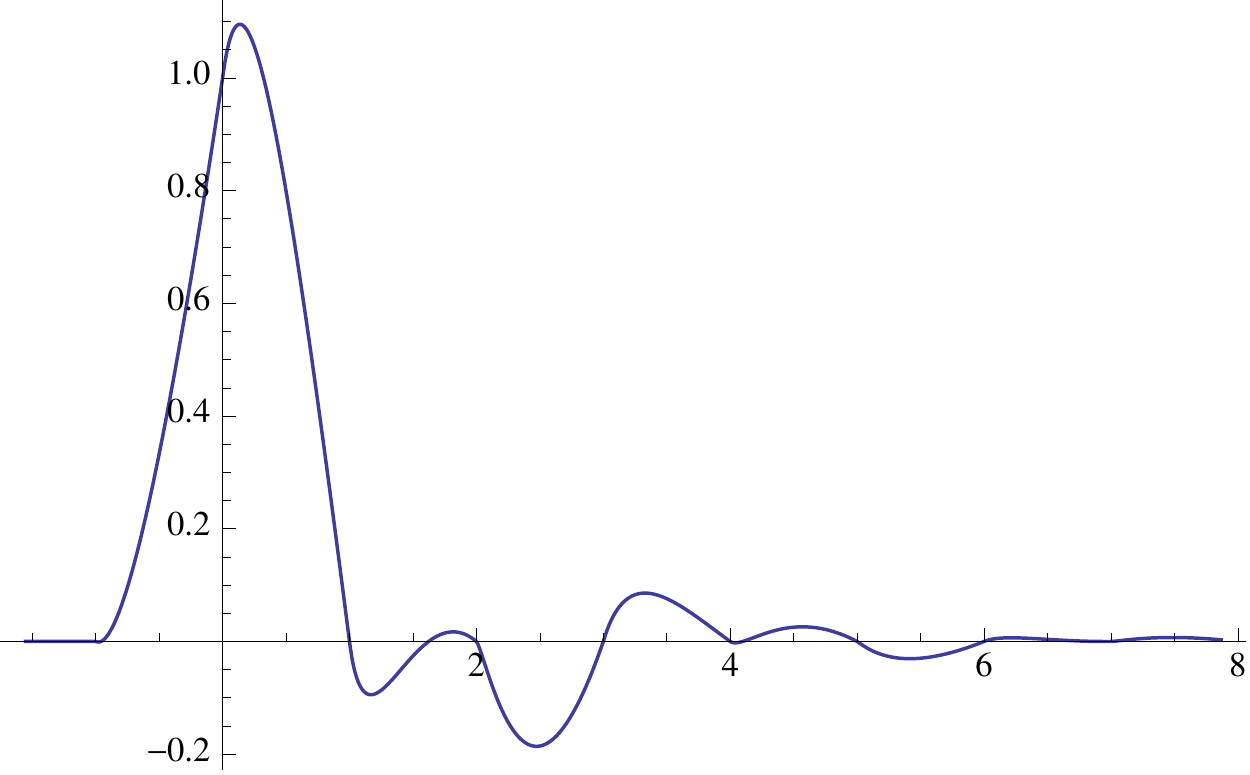}
\caption{The scalar parts of the fundamental cardinal splines $L_{q_1}$ (left) and $L_{q_2}$ (right).}\label{fig1}
\end{center}
\end{figure}

\begin{figure}[h!]
\begin{center}
\includegraphics[width=4cm, height= 3cm]{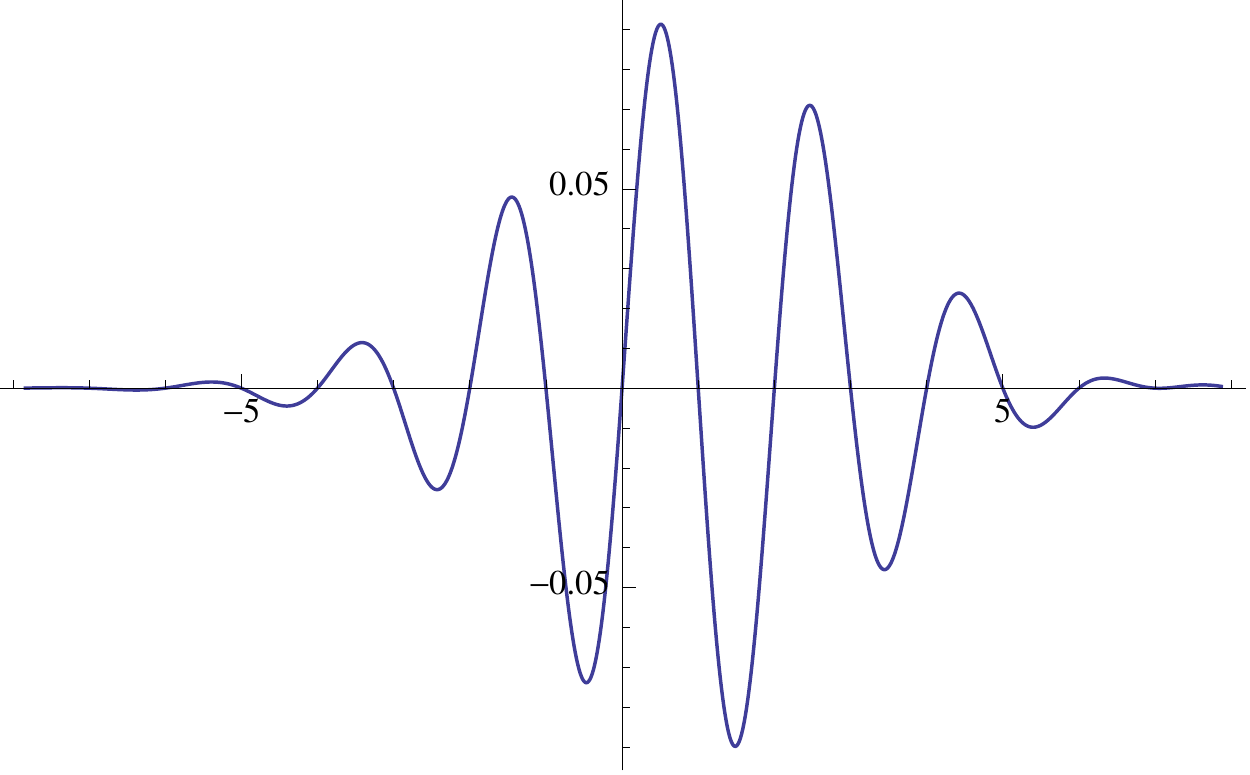}\hspace*{1.5cm}
\includegraphics[width=4cm, height= 3cm]{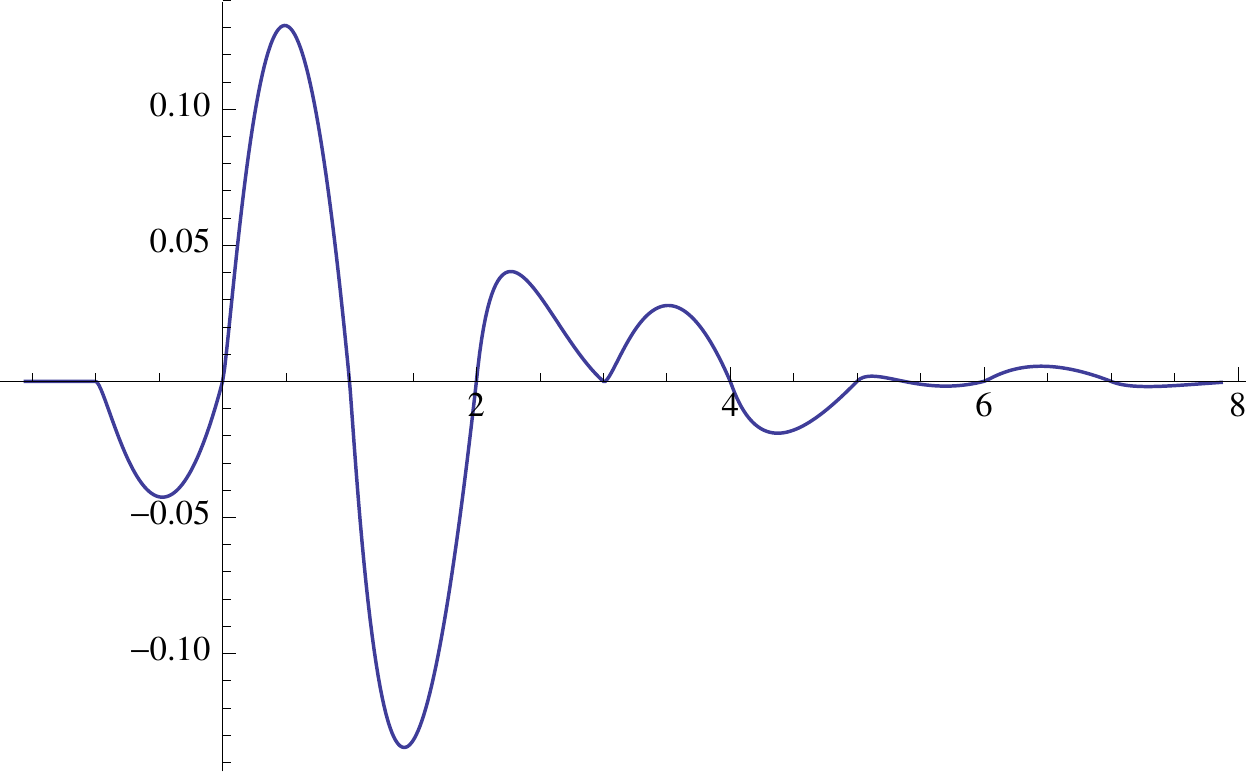}
\caption{The $e_1$-vector part of the fundamental cardinal splines $L_{q_1}$ (left) and $L_{q_2}$ (right)}\label{fig2}
\end{center}
\end{figure}

\begin{figure}[h!]
\begin{center}
\includegraphics[width=4cm, height= 4cm]{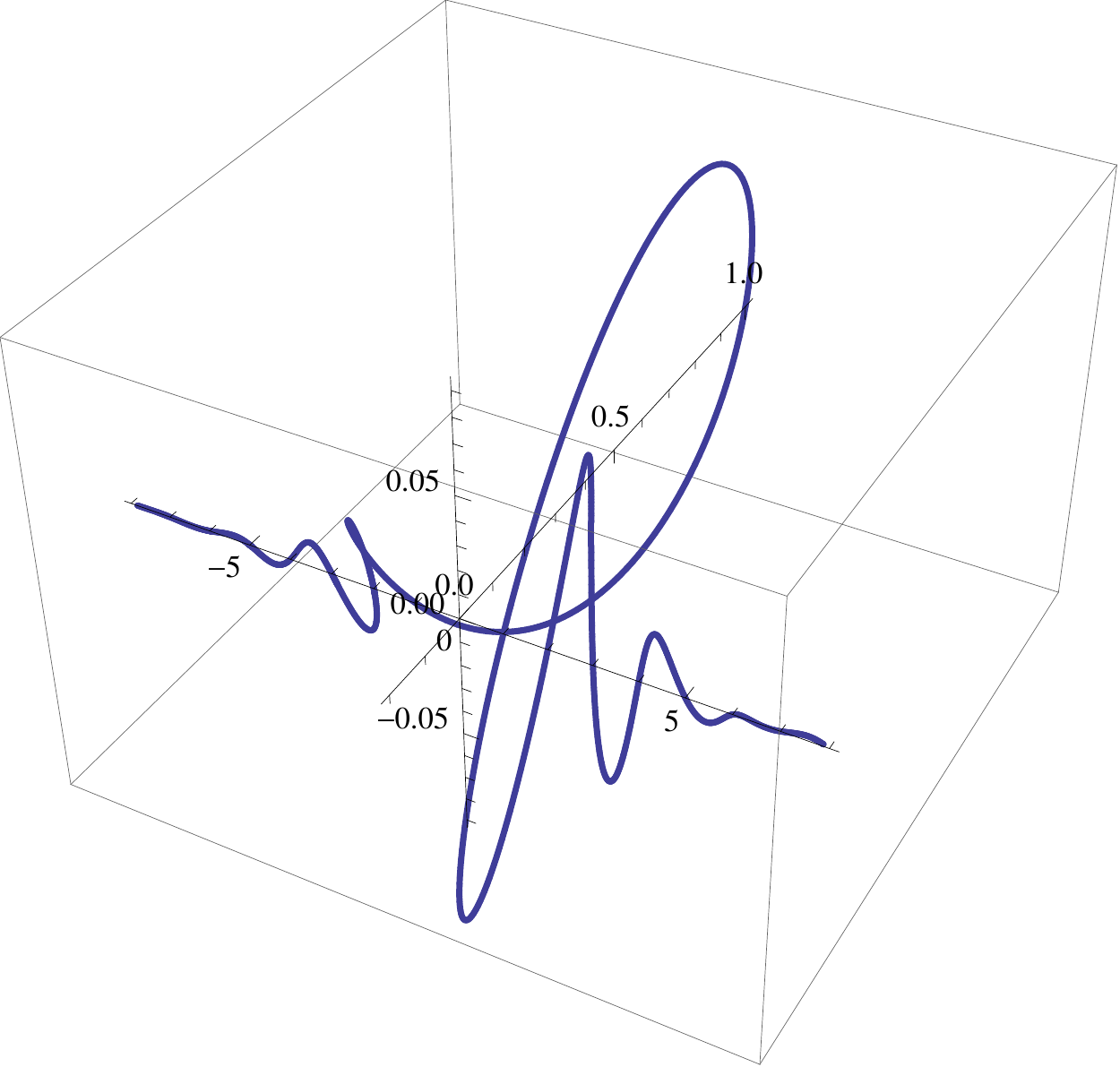}\hspace*{1.5cm}
\includegraphics[width=4cm, height= 4cm]{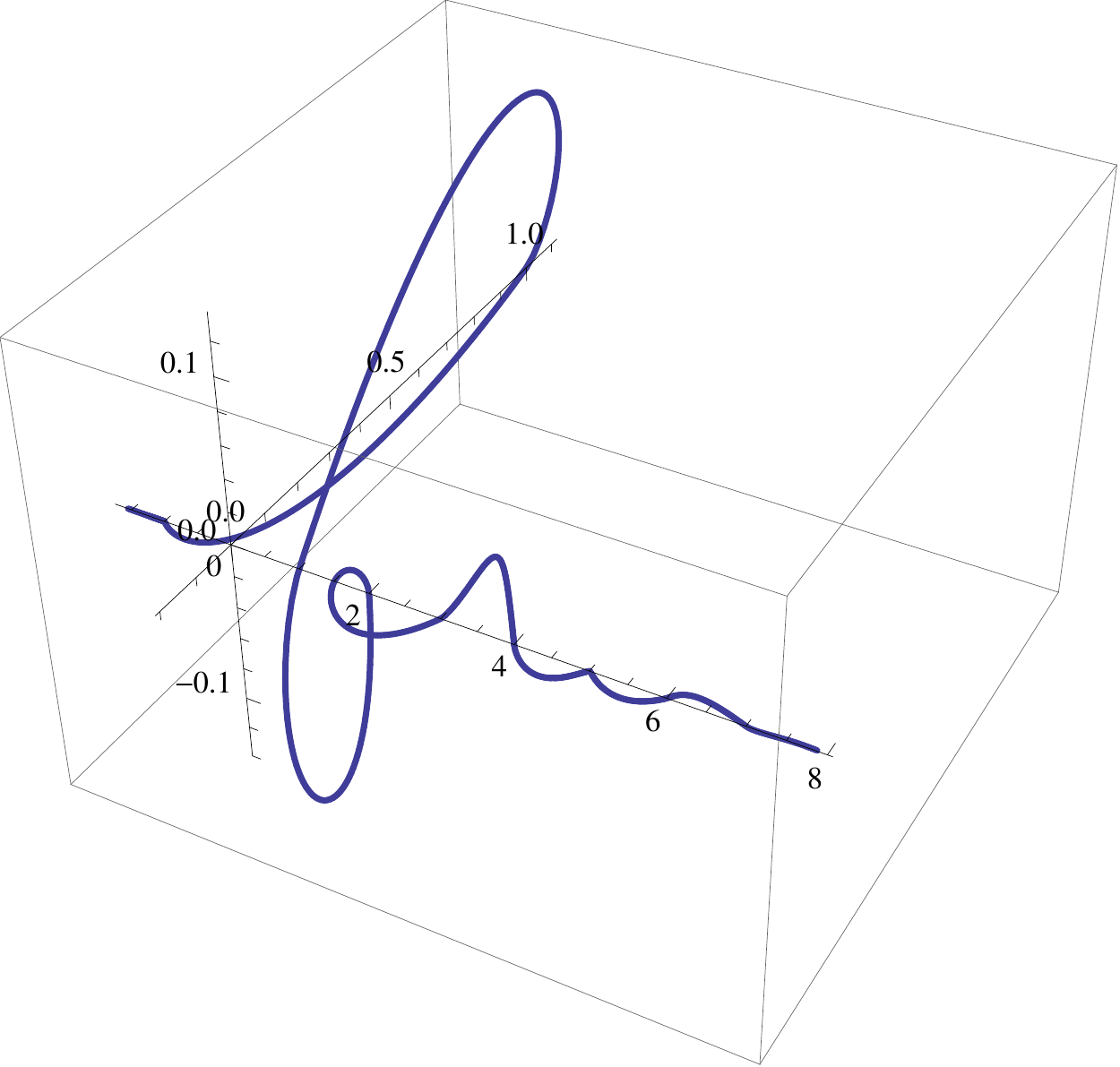}
\caption{The scalar part of the fundamental cardinal splines $L_{q_1}$ (left) and $L_{q_2}$ (right) versus the $e_1$-vector part.}\label{fig5}
\end{center}
\end{figure}

We can use \eqref{compfundspline} to obtain an estimate on the decay of the coefficients $c_k^{(q)}$. The arguments are based on those given in \cite{schoen1,schoen2} for the case $q\in \N$. Equations \eqref{compint2}, \eqref{complexint2}, and \eqref{compfundspline} imply that
\begin{equation}
\sum\limits_{k\in \Z} c_k^{(q)} e^{i k \xi} = \frac{1}{\sum\limits_{k\in \Z}\,B_q (k)\,z^k} =: \phi_q(\xi).\label{c FS}
\end{equation}
The right-hand side $\phi_q$ of the above equation is a function of $z=e^{i\xi}$ which does not have any poles on the unit circle $|z|=1$ provided $q\in Q_R$. As 
$\wh{B}_q\in C^{\lfloor\text{Sc\,}q\rfloor +1}(\R)$ it follows that the Fourier coefficients of $\phi_q(\xi)$ satisfy
\be\label{ckest}
|c_k^{(q)}| \leq M_q\ |k|^{-\lfloor\Sc q\rfloor-1},
\ee
for some positive constant $M_q$. 
The estimate \eqref{ckest} implies the following result.

\begin{proposition}
The fundamental cardinal spline $L_q$ with $q\in Q_R$ satisfies the following point-wise estimate:
\be\label{Lqest}
|L_q (x)| \leq A_q\, |x|^{-\lfloor\Sc q\rfloor}, \quad x\in \R,
\ee 
for some positive constant $A_q$, provided $\Sc q>1$.
\end{proposition}
\begin{proof}
As $\wh{L}_q(-\xi )=\wh{L}_q (\xi )$ and therefore $L_q(-x)=L_q(x)$, it suffices to consider $x \leq 0$. 
From \eqref{compint2}, we obtain
\begin{equation*}|L_q(x)| =\bigg|\sum_{k=-\infty}^{\lfloor x\rfloor}c_q^kB_q(x-k)\bigg| \leq K_q\sum_{k=-\infty}^{\lfloor x\rfloor}|k|^{-\lfloor\Sc q\rfloor-1} \leq A_q\, |x|^{-\lfloor\Sc q\rfloor},
\end{equation*}
where in the inequality we used the fact that the quaternionic B-splines are bounded above by some positive constant $K_q$.
This yields the statement.
\end{proof}
As a direct corollary of Theorem \ref{thm1}, we now provide  a structure theorem for the fundamental cardinal quaternionic splines $L_q$,  relating them to the fundamental cardinal complex splines $B_z$.
\begin{corollary}
Suppose $q = a + v\in Q_R$ and $z=a+i|v|\in{\mathbb C}$. Then
\[
L_q(t)=\Re L_z(t)+\frac{v}{|v|}\Im L_z(t),\qquad t\in \R.
\]
\end{corollary}
\begin{proof} Let $c_j=\Re w_j + \dfrac{v}{|v|}\Im w_j$ where $w_j$ is the solution of $\sum\limits_{j\in \Z}B_z(i-j)w_j=\delta_{i,0}$, $i\in \Z$. Then, for $t\in \R$, 
\begin{align*}
L_q(t)&=\sum_{j\in \Z}B_q(t-j)c_j\\
&=\sum_{j\in \Z}\bigg[\Re(B_z(t-j))+\frac{v}{|v|}\Im (B_z(t-j))\bigg]\bigg[\Re (w_j)+\frac{v}{|v|}\Im (w_j)\bigg]\\
&=\sum_{j\in \Z}\bigg[\Re (B_z(t-j))\Re (w_j)-\Im (B_z(t-j))\Im (w_j)\\
& \quad +\frac{v}{|v|}[\Im (B_z(t-j))\Re (w_j)+\Re (B_z(t-j))\Im (w_j)]\bigg]\\
&=\sum_{j\in \Z}\Re (B_z(t-j)w_j)+\frac{v}{|v|}\sum_{j\in \Z}\Im (B_z(t-j)w_j)\\
&=\Re \left(\sum_{j\in \Z}B_z(t-j)w_j\right)+\frac{v}{|v|}\Im \left(\sum_{j\in \Z}B_z(t-j)w_j\right)\\
& =\Re L_z(t) + \frac{v}{|v|}\Im L_z(t).\qedhere
\end{align*}
\end{proof}
We provide now an alternative computation of the fundamental splines $L_q$ for which error estimates are available.  Starting with (\ref{c FS}) and applying the Poisson summation formula, we write
\begin{equation}
c_k^{(q)}=\frac{1}{2\pi}\int_0^{2\pi}\phi_q(\xi )e^{-ik\xi}\, d\xi =\int_0^{2\pi}\frac{e^{ik\xi}}{\sum\limits_j\widehat B_q(\xi+2\pi j)}\, d\xi .\label{c_k^q int}
\end{equation}
The sum in (\ref{c_k^q int}) is approximated by the truncation
$$F_q^M(\xi )=\sum_{k=-M}^M\widehat B_q(\xi +2\pi k)=\sum_{k=-M}^M\left(\frac{1-e^{-i\xi}}{i(\xi +2\pi k)}\right)^q$$
for suitably large $M > 0$. We have chosen $q$ so that the $2\pi$-periodic function $F_q(\xi )=\sum\limits_k\widehat B_q(\xi +2\pi k)$ has no zeroes and $M$ must be large enough so $F_q^M$ also has no zeroes on the real  line. Then
\begin{align*}
|F_q(\xi )-F_q^M(\xi )|&\leq |1-e^{-i\xi}|^a\sum_{|k|>M}\frac{1}{|\xi +2\pi k|^a}\\
&\leq\frac{2|1-e^{-i\xi}|^a}{(2\pi )^a}\sum_{k=M}^\infty k^{-a}\leq\frac{2}{\pi^a(a-1)M^{a-1}}
\end{align*}
and we conclude that 
\begin{equation}
\bigg|\frac{1}{F_q(\xi )}-\frac{1}{F_q^M(\xi )}\bigg|=\mathcal{O}(M^{1-a}) \quad\text{as $M\to \infty$}.\label{O(M)}
\end{equation}
The coefficients $c_k^{(q)}$, obtained via (\ref{c_k^q int}) are to be estimated from the discrete Fourier transform
$$c_{N,M,k}^{(q)}=\frac{2\pi}{N}\sum_{j=0}^{N-1}\frac{e^{-2\pi ijk/N}}{F_q^M(2\pi j/N)}$$
for $N$ sufficiently large.
We then have
\begin{align}
|c_k^{(q)}-c_{N,M,k}^{(q)}|&=\bigg|\int_0^{2\pi}\frac{e^{-ik\xi}}{F_q(\xi )}\, d\xi-\frac{2\pi}{N}\sum_{j=0}^{N-1}\frac{e^{-2\pi ijk/N}}{F_q^M(2\pi j/N)}\bigg|\notag\\
&\leq \bigg|\int_0^{2\pi}\frac{e^{-ik\xi}}{F_q(\xi )}\, d\xi-\frac{2\pi}{N}\sum_{j=0}^{N-1}\frac{e^{-2\pi ijk/N}}{F_q(2\pi j/N)}\bigg|\notag\\
&+\bigg|\frac{2\pi}{N}\sum_{j=0}^{N-1}\frac{e^{-2\pi ijk/N}}{F_q(2\pi j/N)}-\frac{2\pi}{N}\sum_{j=0}^{N-1}\frac{e^{-2\pi ijk/N}}{F_q^M(2\pi j/N)}\bigg| =: A+B.\label{A+B}
\end{align}
Quantity $B$ on the right hand side of (\ref{A+B}) is easily handled with an application of (\ref{O(M)}), which gives $B=\cO(M^{1-a})$, as $M\to\infty$. To estimate $A$, we employ the following simplification of a result due to Epstein \cite{ep}:

\begin{theorem} \label{epstein} Suppose $f$ is a $2\pi$-periodic function on the real line with $\ell$ continuous derivatives $(\ell \geq 1)$, $\widehat f(k)=\int\limits_0^{2\pi}f(\xi )e^{-ik\xi}\, d\xi$ and $\widehat f_{N,k}:=\frac{2\pi}{N}\sum\limits_{j=0}^{N-1}f(2\pi j/N)e^{-2\pi ijk/N}$ is the $N$-point Riemann sum approximation of the integral defining $\widehat f(k)$. Then for $|k|\leq N/2$,
$$|\widehat f(k)-\widehat f_{N,k}|\leq (2\pi )^2\left(\frac{12}{N}\right)^{\ell}\|f^{(\ell)}\|_\infty .$$
\end{theorem}

It is easily checked that for quaternionic functions on the line of the form $f(\xi )=f_s(\xi )+\dfrac{v}{|v|}f_v(\xi )$ with $v\neq 0$ a fixed quaternionic vector and $f_s$, $f_v$ complex-valued, we have $\dfrac{d}{d\xi}(f(\xi)^{-1})=-\dfrac{f'(\xi)}{f(\xi )^{2}}$ at all $\xi $ for which $f(\xi )\neq 0$.  Applying Theorem \ref{epstein} to the estimate of $A$ in (\ref{A+B}) gives
\begin{equation}
A\leq (2\pi )^2\left(\frac{12}{N}\right)\frac{\|(F_q')_s\|_\infty +\|(F_q')_v\|_\infty}{\inf_\xi |F_q(\xi )|^2}=\mathcal{O}(N^{-1}) \quad\text{as $N\to\infty$}.\label{order constants}
\end{equation}
Returning now to (\ref{A+B}), we have
\begin{equation}
|c_k^{(q)}-c_{N,M,k}^{(q)}|=\mathcal{O}(M^{1-a}+N^{-1}) \quad \text{as $M,N\to\infty$}.\label{coefft est}
\end{equation}
%Of course, $F_q$ is infinitely differentiable and the $N^{-1}$ on the right hand side of (\ref{coefft est})  may be replaced by higher powers  $N^{-\ell}$ at the expense of possibly larger multiplicative constants.

It remains then to compute (numerically) estimates for the multiplicative constant $S_q=\dfrac{\|(F_q')_s\|_\infty +\|(F_q')_v\|_\infty}{\inf_\xi |F_q(\xi )|^2}$. By direct calculation, we have
$$F_q'(\xi )=\begin{cases}
-iq/2&\text{ if }\xi\in{\mathbb Z};\\
\dfrac{iq}{1-e^{-i\xi}}[F_{q+1}(\xi )-ie^{-i\xi}F_q(\xi )]&\text{ if }\xi\notin{\mathbb Z},
\end{cases}$$
and this is used (with $q_1=6.2+\frac{1}{2\sqrt 2}e_1-\frac{1}{4}e_2+\frac{1}{4}e_3$ and $q_2=2.5+\frac{1}{4\sqrt 2}e_1+\frac{1}{8}e_2-\frac{\sqrt {13}}{8}e_3$) in Figures \ref{fig6} and \ref{fig7} below. 

\begin{figure}[h!]
\begin{center}
\includegraphics[width=6cm, height= 5cm]{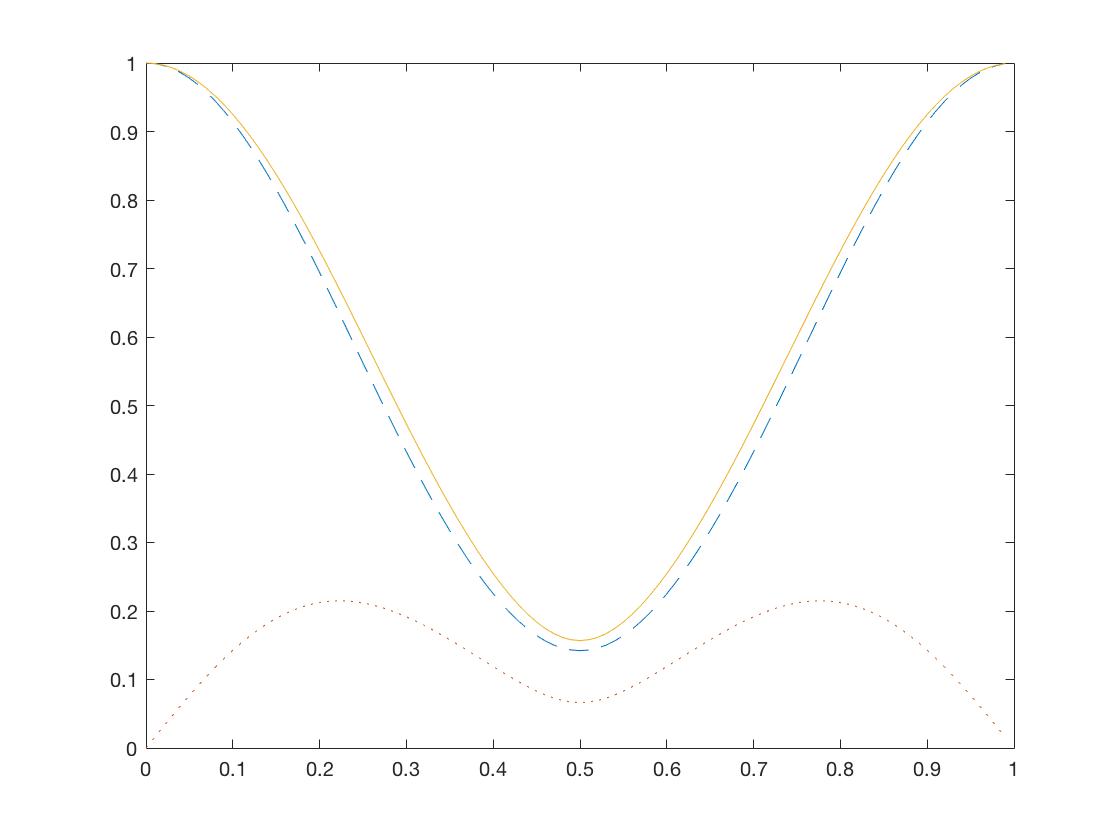}\hspace*{1cm}
\includegraphics[width=6cm, height= 5cm]{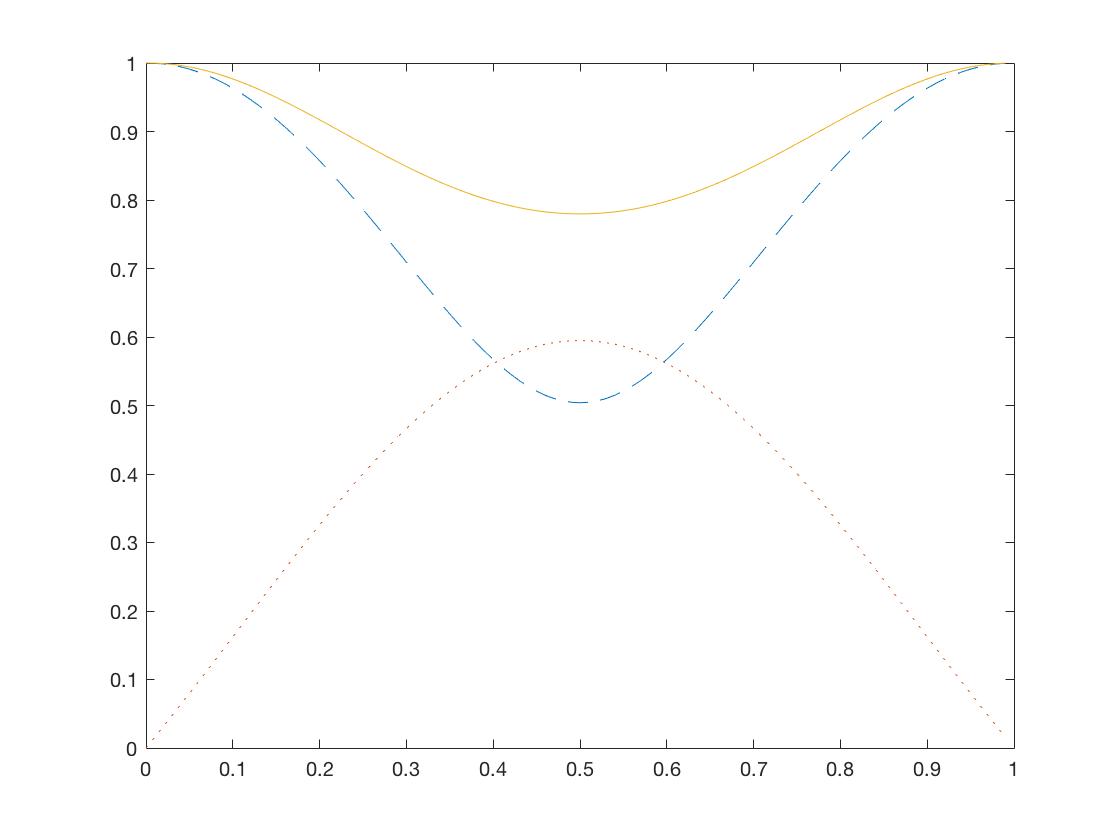}
\caption{Absolute values of the scalar parts (dashed), vector parts (dotted) of the functions  $F_{q_1}$ (left) and $F_{q_2}$ (right). The solid lines are graphs of $|F_{q_1}|$ (left) and $|F_{q_2}|$ (right). }\label{fig6}
\end{center}
\end{figure}

\begin{figure}[h!]
\begin{center}
\includegraphics[width=6cm, height= 5cm]{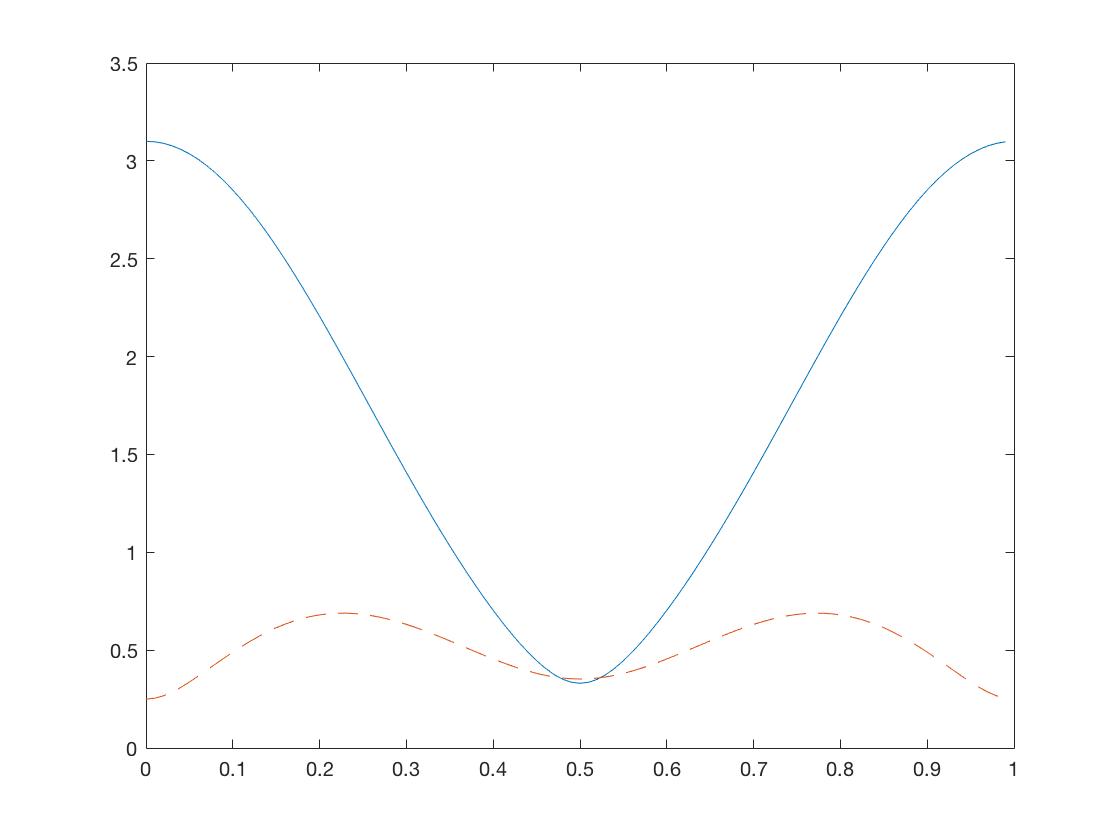}\hspace*{1cm}
\includegraphics[width=6cm, height= 5cm]{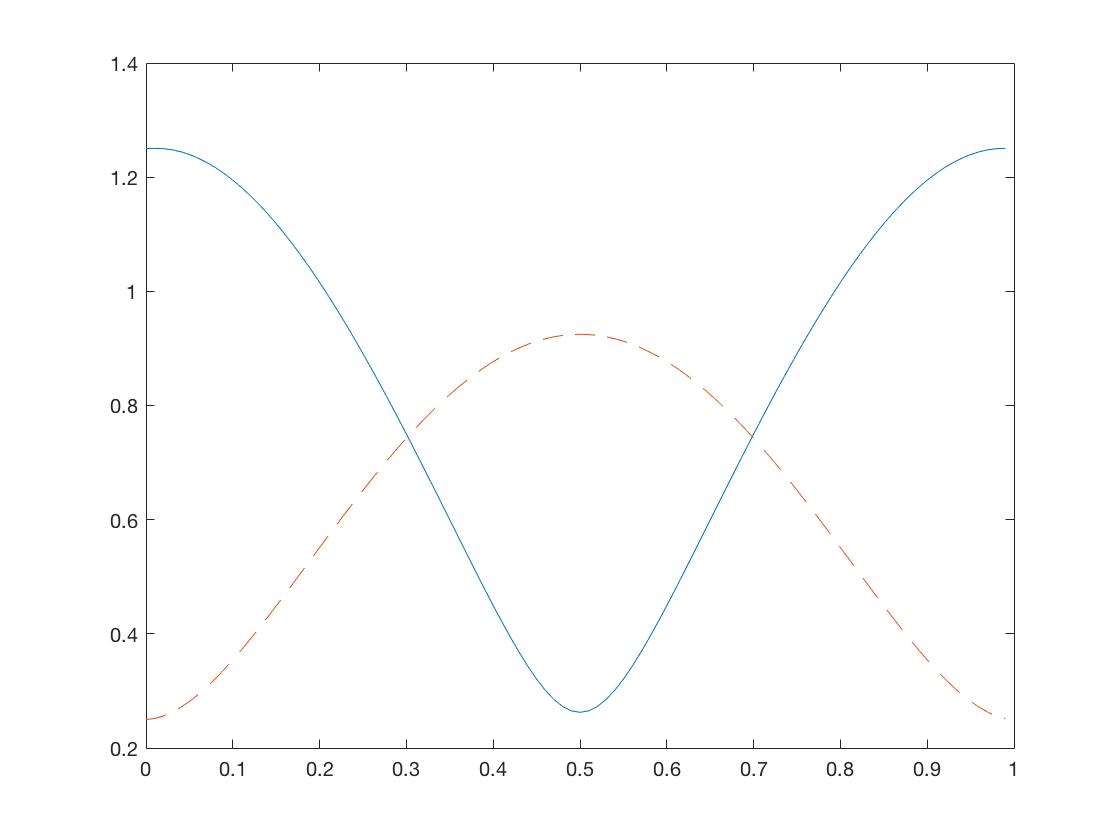}
\caption{Absolute values of the scalar parts (solid) and vector parts (dashed) of $F_{q_1}'$ (left) and $F_{q_2}'$ (right).}\label{fig7}
\end{center}
\end{figure}

These calculations provide numerical estimates of the constants of (\ref{order constants}). We find
$\min |F_{q_1}|\approx 0.1568$, $\min |F_{q_2}|\approx 0.7799$, $\|(F_{q_1})_s'\|_\infty +\|(F_{q_1})_v'\|_\infty\approx 3.7889$,  $\|(F_{q_2})_s'\|_\infty +\|(F_{q_2})_v'\|_\infty\approx 2.1753$.

In these calculations, we have applied Theorem \ref{epstein} with $\ell =1$, but the use of higher derivatives (to obtain faster decay of errors as $N$ increases) is possible, at the expense of larger multiplicative constants in the estimates. This will not be pursued here, but we note that given the calculations of $F_q$ and $F_q'$ already obtained, the second derivative may be obtained from
$$F_q''(\xi )=\begin{cases}
-\dfrac{q(q+1)}{4}&\text{ if $\xi =0$};\\
iq\dfrac{(1-e^{-i\xi})(F_{q+1}'(\xi )-ie^{-i\xi }F_q'(\xi ))+e^{-i\xi}(iF_{q+1}(\xi )-F_q(\xi ))}{(1-e^{-i\xi })^2}&\text{ if $\xi\neq 0$.}
\end{cases}$$
\section{A Sampling Theorem}
In this section, we derive a sampling theorem for functions in the principal shift-invariant space 
\begin{align*}
V_q&=\text{clos}_{\ell^2}\,\text{span}\{B_q(\cdot -n)\}_{n=-\infty}^\infty \\
&=\bigg\{f=\sum_{k}a_kB_q(\cdot -k):\, \{a_k\}_k\in\ell^2({\mathbb Z}),\ \sum_k|a_k|^2<\infty\bigg\}
\end{align*}
where $q\in Q_R$.
 For this purpose, we employ and generalize the following version of Kramer's lemma \cite{K} (which appears in \cite{garcia}) to the complex-quaternionic setting.

\begin{theorem}[\cite{garcia}, p. 501]\label{gensamp}
Let $\emptyset\neq I\subseteq\R$ and let $\{\phi_k: k\in\Z\}$ be an orthonormal basis for $L^2(I)$. Suppose that $\{S_k: k\in \Z\}$ is a sequence of functions $S_k: \Omega\to\C$ on a domain $\Omega\subset\R$ and $\boldsymbol{t} := \{t_k: k\in \Z\}$ a numerical sequence in $\Omega$ satisfying the conditions
\begin{enumerate}
\item[\emph{C1.}]	$S_k(t_l) = a_k \delta_{k,l}$ $(k,l)\in \Z\times \Z$, where $a_k\neq 0$;
\item[\emph{C2.}]	$\displaystyle{\sum_{k\in \Z}}\,\vert S_k(t)\vert^2 < \infty$, for each $t\in \Omega$.
\end{enumerate}
Define a function $K:I\times\Omega \to \C$ by
$K(x,t) := \sum_{k\in \Z} S_k (t) \overline{\phi_k} (x)$
and a linear integral transform $\mcK$ on $L^2 (I)$ by
\[
(\mcK F)(t) := \int_I F(x) K(x,t) \, dx.
\]
Then $\mcK$ is well-defined and injective. Furthermore, if the range of $\mcK$ is denoted by
\[
\mcH := \left\{f:\R\to\C : f = \mcK F, \ F\in L^2(I)\right\},
\]
then
\begin{enumerate}
\item[\emph{(i)}]	$(\mcH, \inn{\mydot}{\mydot}_\mcH)$ is a Hilbert space isometrically isomorphic to $L^2(I)$ ($\mcH \cong L^2(I)$) when endowed with the inner product
\[
\inn{f}{g}_\mcH := \inn{F}{G}_{L^2(I)}
\]
where $f := \mcK F$ and $g = \mcK G$.
\item[\emph{(ii)}] $\{S_k: k\in \Z\}$ is an orthonormal basis for $\mcH$.
\item[\emph{(iii)}] Each function $f\in \mcH$ can be recovered from its samples on the sequence $\boldsymbol{t}$ via the formula
\[
f(t) = \sum_{k\in \Z} f(t_k)\,\frac{S_k (t)}{a_k},
\]
the series converging absolutely and uniformly on subsets of $\R$ where $\|K(\mydot, t)\|_{L^2(I)}$ is bounded.
\end{enumerate}
\end{theorem}

For the proof and further details, we refer the reader to \cite{garcia}.

To carry out the extension to complex quaternions, let $\C_n$ be the complexified Clifford algebra
\[
\C_n :=\{z_Ae_A : z_A\in\C,\, A\subseteq\{1,2,\dots ,n\}\}.
\]
Here, if $A=\{i_1,i_2,\dots ,i_k\}$ with $1\leq i_1<i_2<\cdots <i_k$ then $e_A :=e_{i_1}e_{i_2}\cdots e_{i_k}$ where $e_j^2=-1$ and $e_je_k=-e_ke_j$ for $j\neq k$. We also insist that $e_\emptyset :=1$.

We define an involution on $\C_n$ as the $\C$-conjugate linear mapping $*:\C_n\to\C_n$ by
\begin{equation}
\bigg(\sum_Az_Ae_A\bigg)^* := \sum_A\overline{z_A}\,\overline{e_A}\label{star def},
\end{equation}
where $\overline{z_A}$ is the complex conjugate of the complex number $z_A$ and $\overline{e_A}$ is the Clifford conjugate of the Clifford algebra basis element $e_A$, determined by $\overline{e_j}=-e_j$ $(1\leq j\leq n)$, $\overline{e_0}=e_0$, and $\overline{e_Ae_B}=\overline{e_B}\,\overline{e_A}$. The expression $\sum\limits_{A}$ denotes the sum over all subsets $A\subseteq\{1,2,\dots ,n\}$. 
Note that if $z,w\in\C_n$ then 
$$(zw)^*=w^*z^*.$$
For $z=\sum\limits_Az_Ae_A\in\C_n$, we define $|z|^2=\sum\limits_A|z_A|^2$. Note that $|z|^2=[zz^*]_0=[z^*z]_0=|z^*|^2$ for all $z\in\C_n$.

Let ${\mathcal H}$ be a set on which there are two operations defined:
\begin{itemize}
\item{\it addition}: $+:\cH\times\cH\to \cH$; and
\item{\it scalar multiplication}: $\cdot :\C_n\times\cH\to\cH$. 
\end{itemize}
The triple $(\cH ,+,\cdot )$ is called a  left module over $\C_n$ if  $(\cH ,+)$ is an abelian group and, for all $\lambda ,\mu\in\C_n$ and $x,y\in\cH$, the following conditions are satisfied:
\begin{enumerate}
\item[1.] $\lambda\cdot (x+y)=\lambda\cdot x+\lambda\cdot y$; 
\item[2.] $(\lambda +\mu )\cdot x=\lambda\cdot x+\mu\cdot x$; 
\item[3.] $(\lambda\mu )\cdot x=\lambda\cdot (\mu\cdot x)$; 
\item[4.] $1\cdot x=x$.
\end{enumerate}

\noindent Let $(\cH,+,\cdot )$ be a left module over $\C_n$. Given a mapping $\inn{\mydot}{\mydot}:\cH\times\cH\to\C_n$, we say $(\cH,+,\cdot ,\inn{\mydot}{\mydot} )$ is a left Hilbert module over $\C_n$ if, for all $x,y,z\in\cH$ and $\lambda,\mu\in\C_n$, the following requirements hold:
\begin{enumerate}
\item[5.]  $\langle x+y,z\rangle =\langle x,y\rangle +\langle y,z\rangle$;
\item[6.]  $\langle x,y+z\rangle =\langle x,y\rangle +\langle x,z\rangle$;
\item[7.] $\langle\lambda x,\mu y\rangle=\lambda\langle x,y\rangle\mu^*$;
\item[8.] $\langle y,x\rangle =\langle x,y\rangle^*$;
\item[9.] $\langle x,x\rangle =0\iff x=0$;
\item[10.] $\|x\| := [\langle x,x\rangle ]_0$ defines a pseudo-norm on ${\mathcal H}$ in the sense that
\begin{enumerate}
\item[(a)] $\|x\|\geq 0$ and ($\|x\|=0\iff x=0$);
\item[(b)] $\|x+y\|\leq \|x\|+\|y\|$;
\item[(c)] there exists a constant $C>0$ such that $\|\lambda x\|\leq C\,|\lambda|\,\|x\|$.
\end{enumerate}
\end{enumerate}

\nl

Now suppose that $X$ is a measure space with positive measure $dx$ and
$$L^2(X,\C_n) := \left\{f=\sum_Af_Ae_A : f_A:X\to{\mathbb C}\text{ measurable  and }\|f\|_2^2=\sum_A\|f_A\|_2^2<\infty\right\}.$$
A pairing $\inn{\mydot}{\mydot} :L^2(X,\C_n)\times L^2(X,\C_n)\to\C_n$ is given by
\[
\inn{f}{g} :=\int_Xf(x)g(x)^*\, dx.
\]
With this definition, $L^2(X,\C_n )$ becomes a left Hilbert module over $\C_n$. Note that if $f\in L^2(X,\C_n)$, then
$\left[\int f(x)f(x)^*\, dx\right]_0=\int_X|f(x)|^2\, dx$.
\nl

Let $\cH$ be a left Hilbert module over $\C_n$. A mapping $T:{\mathcal H}\to\C_n$ is called a bounded left linear functional if, for all $f,g\in\cH$ and $\lambda\in\C_n$,
\begin{enumerate}
\item[(i)] $T(f+g)=Tf+Tg$;
\item[(ii)] $T(\lambda f)=\lambda (Tf)$;
\item[(iii)] there exists a constant $M>0$ such that $|Tf|\leq M\|f\|$.
\end{enumerate}
We call $\|T\| :=\inf\{M>0 : \text{ property (iii) is satisfied}\}$ the norm of the functional $T$. 

The next theorem is a generalization of the Riesz Representation Theorem to the complex quaternionic setting.

\begin{theorem}[Riesz Representation Theorem] Let $T$ be a bounded left linear functional on a left Hilbert module ${\mathcal H}$ over $\C_n$ for which there exists an orthonormal basis $\{\varphi_n\}_{n=1}^\infty$, i.e., $\{\varphi_n\}_{n=1}^\infty\subset{\mathcal H}$ and for all $f\in{\mathcal H}$, we have
\begin{enumerate}
\item[\emph{(i)}] $\langle\varphi_n,\varphi_m\rangle =\delta_{n,m}$ for all $m, n \in \N$;
\item[\emph{(ii)}] $\sum_{n=1}^\infty |\inn{f}{\varphi_n} |^2 = \|f\|^2$; 
\item[\emph{(iii)}] $f=\sum\limits_{n=1}^\infty \langle f,\varphi_n\rangle\varphi_n.$
\end{enumerate}
Then there exists a unique $g\in {\mathcal H}$ such that  $Tf=\inn{f}{g}$, for all $f\in {\mathcal H}$.
\end{theorem}

\begin{proof} Let $\{\varphi_n\}_{n=1}^\infty$ be an orthonormal basis for ${\mathcal H}$.
Let $a_j :=T\varphi_j\in\C_n$, and given $f\in{\mathcal H}$, let $c_j :=\langle f,\varphi_j\rangle\in\C_n$ and $f_n :=\sum\limits_{j=1}^nc_j\varphi_j$. Then $\|f-f_n\|_2\to 0$ as $n\to\infty$ and since $T$ is left linear, $Tf_n=\sum\limits_{j=1}^nc_jT\varphi_j=\sum\limits_{j=1}^nc_ja_j$. Also, again by the linearity and boundedness of $T$,
\[
|Tf-Tf_n|=|T(f-f_n)|\leq\|T\|\|f-f_n\|\to 0\text{ as }n\to\infty.
\]
Consequently, we have that $Tf=\sum\limits_{j=1}^\infty c_ja_j$. Furthermore,
\begin{equation}
\left|\left(\sum_{j=1}^n c_ja_j\right)_0\right|\leq\bigg|\sum_{j=1}^nc_ja_j\bigg|=|Tf_n|\leq\|T\|\|f_n\|_2=\|T\| \left(\sum_{j=1}^n |c_j|^2\right)^{1/2}.\label{ac_ineq}
\end{equation}
Putting $c_j=a_j^*$ in (\ref{ac_ineq}) yields
\[
\sum_{j=1}^n|a_j|^2\leq\|T\|\left(\sum_{j=1}^n|a_j|^2\right)^{1/2},
\]
from which we obtain the uniform bound $\bigg(\sum\limits_{j=1}^n|a_j|^2\bigg)^{1/2}\leq \|T\|$. 
Now let $g :=\sum\limits_{j=1}^\infty a_j^*\varphi_j\in{\mathcal H}$. Then 
\[
\langle f,g\rangle=\bigg\langle\sum\limits_{j=1}^\infty c_j\varphi_j,\sum\limits_{\ell=1}^\infty a_\ell^*\varphi_\ell\bigg\rangle=\sum\limits_{j=1}^\infty \sum\limits_{\ell=1}^\infty c_j\langle\varphi_j,\varphi_\ell\rangle a_\ell=\sum\limits_{j=1}^\infty c_ja_j=Tf,
\]
as required. 

To establish uniqueness, suppose that $h$ is another such element of ${\mathcal H}$ for which $Tf=\langle f,h\rangle$, for all $f\in{\mathcal H}$. Then we have
\[
\langle g-h,g-h\rangle = \langle g-h,g\rangle -\langle g-h,h\rangle =T(g-h)-T(g-h)=0,
\]
so that $\|g-h\|^2=[\langle g-h,g-h\rangle ]_0=0$, i.e., $g=h$.
\end{proof}
\nl
The theorem below is the extension of Theorem \ref{gensamp} to the current setting.

\begin{theorem}[Sampling Theorem for $\C_n$-valued Functions]\label{qsampthm}
Let $T$ be a measure space with positive measue $dt$, $S_n:T\to\C_n$ $(n\geq 1)$ a sequence of $\C_n$-valued functions, and $\{t_k\}_{k=1}^\infty\subset T$ such that 
\begin{enumerate}
\item[\emph{(1)}] $S_n(t_k)=\delta_{n,k}$ $(n,k\geq 1)$;
\item[\emph{(2)}] $\sum\limits_{n=1}^\infty |S_n(t)|^2\leq M<\infty$, for all $t\in T$.
\end{enumerate}
Let ${\mathcal H} :=\clos_{\ell^2}\, \Span \{S_n\}=\left\{f=\sum\limits_{n=1}^\infty a_nS_n : a_n\in\C_n\text{ and }\sum\limits_{n=1}^\infty |a_n|^2<\infty\right\}$. Then for all $f\in{\mathcal H}$,
\[
f(t)=\sum_{n=1}^\infty f(t_n)S_n(t)
\]
with convergence in the norm on ${\mathcal H}$.
\end{theorem}
\begin{proof} Let $(X, dx)$ be a measure space and let $\{\phi_n\}_{n=1}^\infty$ be an orthonormal basis for $L^2(X,\C_n )$ in the sense that 
\begin{enumerate}
\item $\int\limits_X\phi_n(x)\phi_m(x)^*\, dx=\delta_{nm}$, for all $n,m\in\N$;
\item $\sum\limits_{n=1}^\infty \left|\int\limits_X F(x)\phi_n(x)^*\, dx\right|^2=\int\limits_X|F(x)|^2\, dx$, for all $F\in L^2(X,\C_n)$;
\item $\sum\limits_{n=1}^\infty \left(\int\limits_X F(x)\phi_n(x)^*\, dx\right)\phi_n=F$, for all $F\in L^2(X,\C_n)$.
\end{enumerate}
Let $K:X\times T\to\C_n$ be given by $K(x,t) :=\sum\limits_{n\in \N}\phi_n(x)^*S_n(t)$ and define an integral operator ${\mathcal K}$ by 
\[
{\mathcal K}F(t) :=\int\limits_XF(x)K(x,t)\, dx,\qquad F\in L^2(X,\C_n).
\]
Note that 
\begin{align*}
\int_X|K(x,t)|^2\, dx&=\int_X\bigg|\sum\limits_{n\in \N}\phi_n(x)^*S_n(t)\bigg|^2\, dx\\
&=\bigg[\int_X\bigg(\sum\limits_{n\in \N}\phi_n(x)^*S_n(t)\bigg)^*\bigg(\sum\limits_{m\in \N}\phi_m(x)S_m(t)\bigg)\, dx\bigg]_0\\
&=\sum\limits_{n,m\in \N}\bigg[S_n^*(t)\int_X\phi_n^*(x)\phi_m(x)\, dx \,S_m(t)\bigg]_0\\
&=\sum\limits_{n\in \N}\left[S_n(t)^*S_n(t)\right]_0=\sum\limits_{m\in \N} |S_n(t)|^2\leq M<\infty.
\end{align*}
If $f={\mathcal K}F$, for some $F\in L^2(X,\C_n)$, then 
\begin{align}
|f(t)|&=\bigg|\int_XF(x)K(x,t)\, dx\bigg|\notag\\
&\leq C\int_X|F(x)||K(x,t)|\, dx\notag\\
&\leq C\left(\int_X|F(x)|^2\, dx\right)^{1/2} \left(\int_X|K(x,t)|^2\, dx\right)^{1/2}\leq C\sqrt M\|F\|_{L^2(X)}\label{eval}.
\end{align}
Hence, the mapping ${\mathcal K}:F\to f$ is well-defined. ${\mathcal K}$ is also one-to-one since if ${\mathcal K}F(t)=0$ for all $t\in T$, 
 then
 \[
 {\mathcal K}F(t_k)=\int_XF(x)K(x,t_k)\, dx=\int_XF(x)\phi_k(x)^*\, dx=0\quad\text {for all } k\in \N,
 \]
 which implies that $F\equiv 0$. 
 
 Now, let ${\mathcal H} :=\text{Ran}({\mathcal K})$ and define a pairing $\inn{\mydot}{\mydot}_{\mathcal H}$ on ${\mathcal H}$ by
\[
\inn{f}{g}_{\mathcal H} :=\int_XF(x)G(x)^*\, dx\in\C_n,
\]
where $F,G\in L^2(X,\C_n)$ are the unique elements for which ${\mathcal K}F=f$ and ${\mathcal K}G=g$. 
We claim that with this pairing ${\mathcal H}$ becomes a left Hilbert module over $\C_n$. To prove this, we need to verify that axioms 1--9 are satisfied. First, notice that the addition and scalar multiplication in ${\mathcal H}$ are defined in the usual way: linearity of the operator ${\mathcal K}$ impies that if ${\mathcal K}F=f$ and ${\mathcal K}G=g$, then $\lambda\cdot f={\mathcal K}(\lambda\cdot F)$ and ${\mathcal K}(F+G)=f+g$. Thus, 
\[
(\lambda\cdot f)(t)=\int_X(\lambda\cdot F(x))K(x,t)\,dx=\lambda\int_XF(x)K(x,t)\, dx=\lambda (f(t))
\]
and
 \begin{align*}
 (f+g)(t)&=\int_X(F+G)(x)K(x,t)\, dx=\int_X[F(x)+G(x)]K(x,t)\, dx\\
 &=\int_XF(x)K(x,t)\, dx+\int_XG(x)K(x,t)\, dx=f(t)+g(t).
 \end{align*}
 Axioms 1--4 follow immediately. Note that if $f,g$ are as above and $h={\mathcal K}H$ then 
 \begin{align*}
 \langle f+g,h\rangle_{\mathcal H}&=\int_X(F+G)(x)H(x)^*\, dx=\int_XF(x)H(x)^*\, dx+\int_XG(x)H(x)^*\, dx\\
 &=\langle f,g\rangle_{\mathcal H}+\langle g,h\rangle_{\mathcal H},
 \end{align*}
and hence axiom 5 is satisfied. Axiom 6 follows similarly. Next, observe that if $f,g$ are as above and $\lambda ,\mu\in\C_n$, then 
\[
\langle\lambda f,\mu g\rangle_{\mathcal H}=\int_X(\lambda F)(x)(\mu G)(x)^*\, dx=\lambda\int_XF(x)G(x)^*\, dx\ \mu^*=\lambda\langle f,g\rangle_{\mathcal H}\ \mu^*,
\]
thus verifying axiom 7. For axiom 8, note that 
\[
\langle g,f\rangle_{\mathcal H}=\int_XG(x)F(x)^*\, dx=\left(\int_XF(x)G(x)^*\, dx\right)^*=\langle f,g\rangle_{\mathcal H}^*.
\]
Suppose now that $\langle f,f\rangle_{\mathcal H}=0$. Then 
\[
\int_X|F(x)|^2\, dx=\bigg[\int_XF(x)F(x)^*\, dx\bigg]_0=[\langle f,f\rangle_{\mathcal H}]_0=0,
\]
so that $F\equiv 0$. Hence $f={\mathcal K}F\equiv 0$ and axiom 9 is verified. Axiom 10(a) now follows immediately. For axiom 10(b), note that 
 \begin{align*}
 \|f+g\|_{\mathcal H}&=[\langle f+g,f+g\rangle_{\mathcal H}]_0^{1/2}\\
 &=\bigg[\int_X(F+G)(x)(F+G)(x)^*\, dx\bigg]_0^{1/2}\\
 &=\left(\int_X|F(x)+G(x)|^2\, dx\right)^{1/2}\\
 &=\|F+G\|_2\leq\|F\|_2+\|G\|_2
 =\|f\|_{\mathcal H}+\|g\|_{\mathcal H}.
 \end{align*}
 Finally, we have
 $$\|\lambda f\|_{\mathcal H}^2=[\langle\lambda f,\lambda f\rangle_{\mathcal H}]_0=\int_X|\lambda F(x)|^2\, dx\leq C^2|\lambda |^2\int_X|F(x)|^2\, dx=C^2|\lambda |^2\|f\|_{\mathcal H}^2,$$
so that axiom 10(c) is verified. 
\nl 
We need to show that ${\mathcal H} :=\text{Ran}({\mathcal K}) = \clos_{\ell^2}\, \Span \{S_n\}$. Since $\{\phi_n\}_{n=1}^\infty$ is an orthonormal basis for $L^2(X,\C_n)$, for each $F\in L^2(X,\C_n)$ we have $F=\sum\limits_{n=1}^\infty a_n\phi_n$ for some sequence $\{a_n\}_{n=1}^\infty\in \ell^2({\mathbb N})$. Then ${\mathcal K}F=\sum\limits_{n=1}^\infty a_nS_n\in\clos_{\ell^2}\, \Span \{S_n\}$. This gives the equivalence of the definitions of ${\mathcal H}$. 
\nl 
We remark that $\{S_n\}_{n=1}^\infty$ is an orthonormal basis for  ${\mathcal H}$. To see this, note that since $S_n={\mathcal K}\phi_n$, we have $\langle S_n,S_m\rangle_{\mathcal H}=\int\limits_X\phi_n(x)\phi_m(x)^*\, dx=\delta_{n,m}$. Also, if $\langle f,S_n\rangle_{\mathcal H}=0$, for all $n\in \N$, then 
\[
\int_XF(x)\phi_n(x)^*\, dx=\langle f,S_n\rangle_{\mathcal H}=0\text{ for all } n\quad\Longrightarrow\quad F\equiv 0\quad\Longrightarrow\quad f\equiv 0.
\]
Moreover,
\[
\sum_{n\in \N}|\langle f,S_n\rangle_{\mathcal H}|^2=\sum_{n\in \N}\left|\int_XF(x)\phi_n(x)^*\, dx\right|^2=\int_X|F(x)|^2\, dx=\|f\|_{\mathcal H}^2.
\]
Equation \eqref{eval} may be interpreted as giving the boundedness of the evaluation functional $E_t:{\mathcal H}\to\C_n$ defined by $E_t(f) := f(t)$. By the Riesz Representation Theorem we conclude that there exists a $k_t\in{\mathcal H}$ such that 
\[
f(t)=E_t(f)=\langle f,k_t\rangle_{\mathcal H}.
\]
Let $k(t,s)=\langle k_t,k_s\rangle_{\mathcal H}=k_s(t)$. Then
\[
\langle f,k(\mydot ,s)\rangle_{\mathcal H}=\langle f,k_s\rangle_{\mathcal H}=f(s).
\]
Hence, $k(t,s)$ is a reproducing kernel for ${\mathcal H}$. Suppose $k'$ is another such kernel. With $k'_s(t)=k'(s,t)$, we have that
\[
k'_s(t)=\langle k'_s,k_t\rangle_{\mathcal H}=\langle k_t,k'_s\rangle^*_{\mathcal H}=k_t(s)^*=\langle k_t,k_s\rangle^*=\langle k_s,k_t\rangle_{\mathcal H}=k_s(t),
\]
i.e., $k'(s,t)=k(s,t)$ for all $s,t\in T$. As $\{S_n\}_{n=1}^\infty$ is an orthonormal basis for ${\mathcal H}$, we obtain
\[
k_t=\sum_n\langle k_t,S_n\rangle_{\mathcal H}S_n=\sum_n\langle S_n,k_t\rangle^*_{\mathcal H}S_n=\sum_nS_n(t)^*S_n,
\]
so that 
\begin{align*}
k(t,s)=\langle k_s,k_t\rangle_{\mathcal H}&=\bigg\langle\sum_{n\in \N} S_n(s)^*S_n,\sum\limits_{m\in \N}S_m(t)^*S_m\bigg\rangle_{\mathcal H}\\
&=\sum\limits_{n,m\in \N}S_n(s)^*\langle S_n,S_m\rangle_{\mathcal H}S_m(t)=\sum\limits_{n\in \N}S_n(s)^*S_n(t).
\end{align*}
Note also that 
\begin{align*}
\int_X K(x,s)^*K(x,t)\, dx&=\int_X\left(\sum\limits_{n\in\N}\phi_n(x)^*S_n(s)\right)^*\left(\sum\limits_{m\in\N}\phi_m(x)^*S_m(t)\right)\, dx\\
&=\sum\limits_{n,m\in \N}S_n(s)^*\int_X\phi_n(x)\phi_m(x)^*\, dx\, S_m(t)\\
&=\sum\limits_{n\in\N}S_n(s)^*S_n(t)=k(t,s).
\end{align*}
Finally, as $K(x,t_n)=\phi_n(x)^*$, we obtain 
\begin{equation}
\langle f,S_n\rangle_{\mathcal H}=\int_X F(x)\phi_n(x)^*\, dx=\int_X F(x) K(x,t_n)\, dx=f(t_n),\label{sampling}
\end{equation}
and therefore, since $\{S_n\}_{n=1}^\infty$ is an orthonormal basis for ${\mathcal H}$, (\ref{sampling}) yields
\[
f=\sum\limits_{n\in\N}\langle f,S_n\rangle_{\mathcal H}S_n=\sum\limits_{n\in\N}f(t_n)S_n.\qedhere
\]
\end{proof}

\begin{remark}
The above results continue to hold when the index set $\N$ is replaced by any countably infinite set.
\end{remark}

For our purposes we choose $T := \R$, $\{t_k\}_{k\in \Z} := \Z$, and for the interpolating function, $S_{k} := L_{q}(\mydot - k)$, $q\in Q_R$. Then Theorem \ref{qsampthm} implies the next result.

\begin{theorem}\label{S Abtastsatz}
Let ${\mathcal H} :=\clos_{\ell^2}\, \Span \{B_q (\mydot - k)\}$ where $\Sc q>1$ and $q\in Q_R$. Then, for all $f\in{\mathcal H}$,
\[
f(t)=\lim_{N\to\infty}\sum_{k=-N}^N f(k)L_q (t-k)
\]
where the convergence is in the $L^2({\mathbb R})$ norm, pointwise, and uniform on compact sets.
\end{theorem}
\begin{proof}
We first show that ${\mathcal H}={\mathcal H}':=\clos_{\ell^2}\, \Span \{L_q (\mydot - k)\}$. Let $f\in{\mathcal H}$, i.e., $f(t)=\sum\limits_kd_kB_q(t-k)$ with $\{d_k\}_{k\in{\mathbb Z}}\in\ell^2({\mathbb Z})$. Then 
\begin{align*}
f(t)&=\sum_kd_k\sum_\ell B_q(\ell )L_q(t-\ell)\\
&=\sum_\ell\left(\sum_kd_kb^q_{\ell -k}\right)L_q(t-\ell )=\sum_k(d*b^q)_\ell L_q(t-\ell)
\end{align*}
where $b^q_k=B_q(k)$ $(k\in{\mathbb Z})$. Since $\{d_k\}_{k\in{\mathbb Z}}\in\ell^2({\mathbb Z})$ and $b^q\in\ell^1({\mathbb Z})$, an application of Young's inequality gives
$$\|d*b^q\|_{\ell^2}\leq C\|d\|_{\ell^2}\|b^q\|_{\ell^1}$$
from which we see that $f\in{\mathcal H}'$ and  consequently ${\mathcal H}\subset{\mathcal H}'$.
%\vskip0.2in
%\noindent\textcolor{red}{We need $b^q\in\ell^1({\mathbb Z})$. 
% What is the best estimate of the decay of $B_q$?} 
% \vskip0.2in
%\noindent\textcolor{blue}{As the decay depends only on the modulus, it is the same as for the fractional B-splines. Therefore, 
%\[
%B_q (x) \in \mathcal{O}(|x|^{-(\Sc q +1)}), \quad\text{for $\Sc q > 1$ as $x\to\infty$}. 
%\]
%}
% \vskip0.2in
Note that the function $\sum\limits_k\widehat{B_q}(\xi +2\pi k)$ has absolutely convergent Fourier series $\sum\limits_\ell B_q(\ell )e^{i\ell\xi}$ (since  $\Sc q_1$). Also, $\sum\limits_k\widehat{B_q}(\xi +2\pi k)$ is zero-free, and we conclude from Wiener's Tauberian theorem \cite{katz} that $\dfrac{1}{\sum\limits_k\widehat{B_q}(\xi +2\pi k)}$ has absolutely convergent Fourier series, i.e., 
\[
\dfrac{1}{\sum\limits_k\widehat{B_q}(\xi +2\pi k)}=\sum_kc_k^qe^{ik\xi}
\]
with $\sum\limits_k|c_k^q|<\infty$. Suppose $f\in{\mathcal H}'$, i.e., $f(t)=\sum\limits_kd_kL_q(t-k)$ with $\{d_k\}_{k=-\infty}^\infty\in\ell^2({\mathbb Z})$. Then 
$$f(t)=\sum_kd_k\sum_jc^q_jB_q(t-j-k)=\sum_j(d*c^q)_jB_q(t-j).$$
But Young's theorem gives
$\|d*c^q\|_{\ell^2}^2\leq C\|c^q\|_{\ell_1}\|d\|_{\ell^2}$,
and therefore $f\in{\mathcal H}$ so that ${\mathcal H}'\subset{\mathcal H}$. We conclude that ${\mathcal H}={\mathcal H}'$.

In the spirit of Theorem 7, let $I=[-\pi ,\pi ]$ and $\phi_n(x)=\dfrac{e^{inx}}{\sqrt{2\pi}}$ $(n\in{\mathbb Z})$. Then $\{\phi_n\}_{n=-\infty}^\infty$ is an orthonormal basis for $L^2(I)$. Let $S_m(t)=L_q(t-m)$ and $t_m=m$ $(m\in{\mathbb Z})$. Then 
$$K(x,t)=\sum_{n=-\infty}^\infty\overline{\phi_n(x)}S_n(t)=\frac{1}{\sqrt{2\pi}}\sum_{n=-\infty}^\infty e^{-inx}L_q(t-n).$$
Therefore, if $F\in L^2(I)$,
$${\mathcal K}F(t)=\frac{1}{\sqrt{2\pi}}\int_0^1F(x)\sum_ke^{-ikx}L_q(t-k)\, dt=\sum_k\hat F(k)L_q(t-k)$$
where $\hat F(k)=\dfrac{1}{\sqrt{2\pi}}\int_0^1F(x)e^{-ikx}\, dx$ is the $k$-th Fourier coefficient of $F$. Furthermore,
$$\widehat{{\mathcal K}F}(\xi )=\sum_k\hat f(k)\widehat{L_q}(\xi)e^{ik\xi}=F(\xi )\widehat{L_q}(\xi ).$$
By the proof of Theorem \ref{qsampthm}, $\text{Ran}({\mathcal K})$ is a Hilbert space with inner product
$$\langle f,g\rangle_{\text{Ran}({\mathcal K})}=\langle{\mathcal K}F,{\mathcal K}G\rangle_{\text{Ran}({\mathcal K})}=\langle F,G\rangle_{L^2[-\pi ,\pi ]}.$$
Let $f={\mathcal K}F$ for some $F\in L^2[-\pi ,\pi ]$. Then 
\begin{align*}
\|f\|_{L^2({\mathbb R})}^2=\|{\mathcal K}F\|_{L^2({\mathbb R})}^2&=\int_{-\infty}^\infty |F(\xi )|^2|\widehat{L_q}(\xi )|^2\,  d\xi\\
&=\int_{-\pi}^{\pi}|F(\xi )|^2\sum_k|\widehat{L_q}(\xi +2\pi k)|^2\, d\xi\\
&\leq C\|F\|_{L^2[-\pi ,\pi]}^2=\|f\|_{\text{Ran}({\mathcal K})}^2
\end{align*}
where 
$$C=\frac{\sup_{|\xi|\leq\pi}\sum_k|\widehat{B_q}(\xi +2\pi k)|^2}{\inf_{|\xi |\leq\pi}\left|\sum_\ell\widehat{B_q}(\xi +2\pi\ell )\right|^2}<\infty$$
since $q\in Q_R$ and $\Sc q>1$.
On the other hand,
\begin{align*}
\|f\|_{L^2({\mathbb R})}^2=\|{\mathcal K}F\|_{L^2({\mathbb R})}^2&=\int_{-\pi}^{\pi}|F(\xi )|^2\sum_k|\widehat{L_q}(\xi +2\pi k)|^2\, d\xi\\
&\geq c\int_{-\pi}^{\pi}|F(\xi )|^2\, d\xi=c\|F\|_{L^2[-\pi,\pi]}^2=c\|f\|_{\text{Ran}({\mathcal K})}^2
\end{align*}
where 
$$c=\frac{\inf_{|\xi |\leq\pi}\sum_k|\widehat{B_q}(\xi +2\pi k)|^2}{\sup_{|\xi |\leq\pi}\left|\sum_\ell\widehat{B_q}(\xi +2\pi\ell )\right|^2}\geq\frac{(2/\pi )^{\Sc q}}{\sup_{|\xi |\leq\pi}\left(\sum_\ell |\widehat{B_q}(\xi +2\pi\ell )|\right)^2}>0.$$
%\begin{align*}
%c&=\inf_{|\xi |\leq\pi}|\widehat{L_q}(\xi )|\\
%&=\inf_{|\xi |\leq\pi}\frac{|\hat B_q(\xi )|}{\left|\sum\limits_{k=-\infty}^\infty \widehat{B_q}(\xi +2\pi k)\right|}\geq C\left(\inf_{|\xi |\leq\pi}\bigg|\sum_{k=-\infty}^\infty\hat B_q(\xi +2\pi k)\bigg|\right)^{-1}>0
%\end{align*}
%since $q\in Q_R$. 
We therefore have 
\begin{equation}
\|f\|_{L^2({\mathbb R})}\simeq\|f\|_{\text{Ran}({\mathcal K})}\simeq\|f\|_{\mathcal H}\label{norm eq}
\end{equation}
 for all $f\in{\mathcal H}$. 
 By Theorem \ref{qsampthm}, each $f\in{\mathcal H}$ admits the sampling expansion 
$f(t)=\sum\limits_{k=-\infty}^\infty f(k)L_q(t-k)$ with convergence in the norm of $\text{Ran}({\mathcal K})$, or equivalently (by (\ref{norm eq})), in the $L^2({\mathbb R})$ norm. By (\ref{eval}), norm convergence implies pointwise convergence and uniform convergence on compact sets  since an application of Young's inequality gives\
\begin{align*}
\sum_{n=-\infty}^\infty |L_q(t-n)|^2&=\sum_{n=-\infty}^\infty\bigg|\sum_{k=-\infty}^\infty c_k^qB_q(t-n-k)\bigg|^2\\
&\leq\sum_{k=-\infty}^\infty |c_k^q|^2\bigg(\sum_{n=-\infty}^\infty |B_q(t-n)|\bigg)^2\\
&\leq C\bigg(\inf_{|\xi |\leq\pi}\bigg|\sum_{k=-\infty}^\infty\widehat{B_q}(\xi +2\pi k)\bigg|\bigg)^{-2}\leq C'<\infty
\end{align*}
where we have used the decay estimate on $B_q$ which is valid since $\Sc q >1$.
\end{proof}

\section{Summary}
We constructed fundamental cardinal B-splines $L_q$ of quaternionic orders $q$ where $q$ belongs to a certain nonempty region in $\bH_\R$. These quaternionic splines satisfy the interpolation conditions $L_q (m) = \delta_{m,0}$, $m\in \Z$. The construction employs interesting properties of an associated quaternionic Hurwitz zeta function and the existence of complex quaternionic inverses. We showed that the cardinal fundamental splines of quaternionic order fit into the setting of Kramer's Lemma and allow for a family of sampling (respectively, interpolation) series.

\section{Acknowledgements}
The authors wish to thank the anonymous referees whose insightful questions have led to significant improvements to this paper.

\end{document}